\documentclass{amsart}
\newcommand{\mt}[1]{\mathtt{#1}}

\newcommand{\mF}{\mathcal{F}}

\newcommand{\bfx}{\mathbf{x}}
\newcommand{\spn}{\mbox{span}}
\newcommand{\bn}{\bar{n}}

\usepackage{amsmath}
\usepackage{amsthm}
\usepackage{graphicx}
\usepackage{color}
\usepackage{amsfonts}
\usepackage{amssymb}
\usepackage{mathrsfs}
\usepackage{amscd}
\usepackage{url}

\newcommand{\on}{\bar{n}}
\newcommand{\re}{\mathbb{R}}
\newcommand{\cpx}{\mathbb{C}}

\newcommand{\N}{\mathbb{N}}

\renewcommand{\P}{\mathbb{P}}

\newcommand{\lmd}{\lambda}

\newcommand{\eps}{\epsilon}

\def\af{\alpha}
\def\bt{\beta}
\def\gm{\gamma}
\def\rank{\mbox{rank}}

\newcommand{\sig}{\sigma}

\newcommand{\reff}[1]{(\ref{#1})}

\newcommand{\mc}[1]{\mathcal{#1}}

\newcommand{\bdes}{\begin{description}}
\newcommand{\edes}{\end{description}}

\newcommand{\bal}{\begin{align}}
\newcommand{\eal}{\end{align}}

\newcommand{\bnum}{\begin{enumerate}}
\newcommand{\enum}{\end{enumerate}}

\newcommand{\bit}{\begin{itemize}}
\newcommand{\eit}{\end{itemize}}

\newcommand{\bea}{\begin{eqnarray}}
\newcommand{\eea}{\end{eqnarray}}
\newcommand{\be}{\begin{equation}}
\newcommand{\ee}{\end{equation}}

\newcommand{\baray}{\begin{array}}
\newcommand{\earay}{\end{array}}

\newcommand{\bsry}{\begin{subarray}}
\newcommand{\esry}{\end{subarray}}

\newcommand{\bca}{\begin{cases}}
\newcommand{\eca}{\end{cases}}

\newcommand{\bcen}{\begin{center}}
\newcommand{\ecen}{\end{center}}

\newcommand{\bbm}{\begin{bmatrix}}
\newcommand{\ebm}{\end{bmatrix}}

\newcommand{\bmx}{\begin{matrix}}
\newcommand{\emx}{\end{matrix}}

\newcommand{\bpm}{\begin{pmatrix}}
\newcommand{\epm}{\end{pmatrix}}

\newcommand{\btab}{\begin{tabular}}
\newcommand{\etab}{\end{tabular}}

\newtheorem{theorem}{Theorem}[section]

\newtheorem{prop}[theorem]{Proposition}
\newtheorem{lem}[theorem]{Lemma}

\newtheorem{cor}[theorem]{Corollary}

\newtheorem{defi}[theorem]{Definition}
\theoremstyle{definition}

\newtheorem{exm}[theorem]{Example}
\newtheorem{alg}[theorem]{Algorithm}

\newtheorem{remark}[theorem]{Remark}

\begin{document}

\title[Low Rank Tensor Approximations]
{Nearly Low Rank Tensors and Their Approximations}

\author{Jiawang Nie}
\address{
Department of Mathematics,  University of California San Diego,  9500
Gilman Drive,  La Jolla,  California 92093,  USA.
} \email{njw@math.ucsd.edu}

\begin{abstract}
A tensor is a multi-indexed array. The low rank tensor approximation problem (LRTAP) is
to find a tensor whose rank is small and that is close to a given one.
This paper studies the LRTAP when the tensor to be approximated
is close to a low rank one. Both symmetric and nonsymmetric tensors are discussed.
We propose a new approach for solving the LRTAP.
It consists of three major stages:
i) Find a set of linear relations that are approximately satisfied by the tensor;
such linear relations can be expressed by polynomials and
can be found by solving linear least squares.
ii) Compute a set of points that are approximately common zeros
of the obtained polynomials;
they can be found by computing Schur decompositions.
iii) Construct a low rank approximating tensor from the obtained points;
this can be done by solving linear least squares.
Our main conclusion is that if the given tensor is sufficiently
close to a low rank one, then the computed tensor
is a good enough low rank approximation. This approach can also be applied to
efficiently compute low rank tensor decompositions,
especially for large scale tensors.
\end{abstract}



\maketitle

\section{Introduction}

Let $m$ and $n_1,\ldots,n_m$ be positive integers. Denote by
$\cpx^{n_1,\ldots,n_m}$ the space of all complex tensors of order $m$
and dimension $(n_1,\ldots,n_m)$.
%
%
Each $\mc{F} \in \cpx^{n_1,\ldots,n_m}$ is
indexed by an integer tuple $(i_1, \ldots, i_m)$ with
$1 \leq i_j \leq n_j \, (j=1,\ldots,m)$, i.e.,
\[
\mc{F} = (\mc{F}_{i_1 \ldots i_m})_{1 \leq i_1 \leq n_1, \ldots,  1 \leq i_m \leq n_m }.
\]
Tensors of order $m$ are called $m$-tensors.
When $m=1$ (resp., $2$), they are vectors (resp., matrices).
When $m=3$ (resp., $4$), they are called cubic (resp., quartic) tensors.
The standard norm of $\mF$ is defined and denoted as:
\be \label{tensor:norm}
\|\mF\|  :=
\Big( \sum_{ 1\leq i_1 \leq n_1, \ldots,  1 \leq i_m \leq  n_m }
|\mc{F}_{i_1 \ldots i_m} |^2 \Big)^{1/2}.
\ee
For vectors $u^1 \in \cpx^{n_1}, \ldots, u^m
\in \cpx^{n_m}$, their outer product $u^1 \otimes \cdots \otimes u^m$
is the tensor in $\cpx^{n_1,\ldots,n_m}$ such that
for all $i_1, \ldots, i_m$ in the range,
\[
(u^1 \otimes \cdots \otimes u^m)_{i_1,\ldots,i_m} = (u^1)_{i_1}
\cdots (u^m)_{i_m}.
\]
An outer product like $u^1 \otimes \cdots \otimes u^m$
is called a rank-1 tensor.
For every $\mc{F}$, there exist rank-1 tensors
$\mc{F}_1,\ldots, \mc{F}_r $ such that
\be \label{dcmp:rank-r}
\mc{F} = \mc{F}_1 + \cdots + \mc{F}_r.
\ee
The smallest such $r$ is called the rank of $\mF$,
denoted as $\rank(\mF)$. If $\rank(\mF)=r$,
$\mF$ is called a rank-$r$ tensor and
\reff{dcmp:rank-r} is called a rank decomposition.
We refer to \cite{KolBad09,Lim13}
for the theory and applications of tensors.

\subsection{Low rank approximations}

Typically, it is hard to compute rank decompositions for tensors (cf.~\cite{HiLi13}).
In applications, we often need to approximate a tensor
by a low rank one. The low rank tensor approximation problem (LRTAP)
is that, for given $\mF \in \cpx^{n_1,\ldots,n_m}$
and given $r$ (typically small), find
$r$ tuples $u^{(\ell)}:=(u^{\ell,1}, \ldots, u^{\ell,m}) \in
\cpx^{n_1} \times \cdots \times \cpx^{n_m}$ ($\ell = 1,\ldots,r$)
that give a minimizer to the nonlinear least squares problem
\be \label{nLS:lra:F}
\min_{ u^{(1)}, \ldots,   u^{(m)} } \quad
\Big\| \sum_{\ell=1}^r  u^{\ell,1} \otimes \cdots \otimes u^{\ell,m}
- \mc{F} \Big\|^2.
\ee
In contrast to the matrix case, the best rank-$r$
tensor approximation may not exist when the order $m>2$ (cf.~\cite{DeSLim08}).
This is because the set of tensors,
whose ranks are less than or equal to $r$,
may not be closed. For such case, an almost best rank-$r$ approximation
is then often wanted.
The LRTAP has important applications in computer vision \cite{VasTer02,WangAhu03},
linear algebra \cite{HOSVD2000,Kol01},
multiway data analysis \cite{Kroon08}, numerical analysis \cite{HacKho07,KhoKho07},
and signal processing \cite{CPdL02,dLaMo98,SBG00}.
A frequently used method for solving LRTAPs is the
alternating least squares (ALS) (cf.~\cite{CLdA09,KolBad09}).
ALS can be easily implemented in computations,
while its convergence property is generally not very satisfying.
When $r=1$, the case of best rank-$1$ approximations,
there exist other methods, e.g., higher order power iterations
(cf.~\cite{LMV00b,SHOPM2002,KolMayo11,ZLQ12}),
semidefinite relaxations (cf.~\cite{NieWan13}).
We refer to \cite{ComLim11,GKT13} for recent work
on low rank tensor approximations. In applications,
a related problem is the multilinear low rank approximation
(cf.~\cite{LMV00b,FriTam14,IAvD13,SavLim10}).

In applications, people often need to approximate symmetric tensors
with low rank symmetric ones.
A tensor $ \mF \in \cpx^{n_1 \times \cdots \times n_m}$ is symmetric
if $n_1=\cdots = n_m=n$ and each $\mF_{i_1 \ldots i_m}$ is invariant
under permutations of $(i_1,\ldots,i_m)$.
Let $\mt{S}^m(\cpx^{n})$ be the space of all symmetric tensors
of order $m$ and dimension $(n,\ldots,n)$.
For $u \in \cpx^n$, its $m$-th tensor power is defined as
\[
u^{\otimes m} := u \otimes \cdots \otimes u \quad
(u \mbox{ is repeated $m$ times}).
\]
Clearly, each $u^{\otimes m}$ is a rank-1 symmetric tensor.
The low rank symmetric tensor approximation problem (LRSTAP) is that,
for given $\mF \in \mt{S}^m(\cpx^{n})$ and given $r$ (typically small), find
vectors $u_1,\ldots, u_r \in \cpx^n$
that give a minimizer to the nonlinear least squares problem
\be \label{nLS:lra:symF}
\min_{ u_1,\ldots, u_r \in \cpx^n } \quad
\Big\| (u_1)^{\otimes m} + \cdots +(u_r)^{\otimes m}
- \mc{F} \Big\|^2.
\ee
When $r=1$, there exists much work on the approximations.
We refer to \cite{LMV00b,SHOPM2002,NieWan13,ZhaGol01,ZLQ12}.
%
%
When $r>1$, there exists relatively few work.
The ALS method for nonsymmetric tensors
would also be used for approximating symmetric tensors,
by forcing the symmetry in the computation (cf.~\cite{LMV00b,SHOPM2002}).
But its performance is generally not very satisfying.
%
%

\subsection{Contributions}

In this paper, we propose a new approach for computing
low rank approximations, for both symmetric and nonsymmetric tensors.
It is based on finding linear relations satisfied by low rank tensors,
which can be expressed by the so-called generating polynomials (cf.~\cite{GPSTD}).
In applications of low rank approximations,
the tensor to be approximated is often close to a low rank one.
This is typically the case in applications, because the occurring errors
(usually caused by measuring inaccuracies or noises) are often small.
This fact motivates us to compute approximations
by finding the hidden linear relations that are
satisfied by low rank tensors.

We propose to use generating polynomials for
computing low rank tensor approximations.
First, we estimate generating polynomials by solving linear least squares.
Second, we find their approximately common zeros;
this can be done by computing Schur decompositions.
Third, we construct a low rank approximation from those zeros,
by solving linear least squares.
Our main conclusion is that if the tensor to be approximated is close enough
to a low rank one, then the constructed tensor
is a good enough low rank approximation.
The proof is build on perturbation analysis of linear least squares
and Schur decompositions. The proposed methods
can also be applied to compute low rank decompositions
efficiently, especially for large scale tensors.

The paper is organized as follows.
In \S \ref{sc:prlm}, we present some basics that need to be used.
In \S \ref{sc:grdef}, we introduce generating polynomials
for both symmetric and nonsymmetric tensors.
In \S \ref{sc:lra:sym}, we give an algorithm for computing
low rank symmetric approximations for symmetric tensors, and then analyze its performance.
In \S \ref{sc:lra:nonsym}, we give an algorithm for computing
low rank approximations for nonsymmetric tensors, and analyze its performance.
In \S \ref{sc:comp}, we report numerical experiments for the proposed algorithms.

\section{Preliminaries}
\label{sc:prlm}
\setcounter{equation}{0}

\noindent
{\bf Notation} \,
The symbol $\N$ (resp., $\re$, $\cpx$) denotes the set of
nonnegative integers (resp., real, complex numbers).
For any $t\in \re$, $\lceil t\rceil$ (resp., $\lfloor t\rfloor$)
denotes the smallest integer not smaller
(resp., the largest integer not bigger) than $t$.
%
%
The cardinality of a finite set $S$ is denoted as $|S|$.
%
%
For a complex matrix $A$, $A^T$ denotes its transpose and $A^*$
denotes its conjugate transpose.
%
%
%
For a complex vector $u$, $\| u \|_2 = \sqrt{u^*u}$ denotes the standard Euclidean norm.
For a matrix $A$, $\| A \|_2$ denotes its standard operator $2$-norm,
and $\|A\|_F$ denotes its standard Frobenius norm.

\subsection{Catalecticant matrices and border ranks}
\label{ssc:Cat:bdrk}

We partition $n_1,\ldots,n_m$
into two disjoint sets $S_1$ and $S_2$ such that
$
\big| \Pi_{i \in S_1} n_i -  \Pi_{i\in S_2} n_i  \big|
$
is minimum. Up to a permutation of indices, we assume that
\[
S_1=\{n_1, \ldots, n_{\hat{m}} \},  \quad
S_2= \{ n_{\hat{m}+1}, \ldots,  n_m \}.
\]
For convenience, let
\[
I_1 = \{
(\imath_1,\ldots, \imath_{m_1}):  \, 1 \leq \imath_j \leq n_j
\,\mbox{ for }\, j=1, \ldots,  \hat{m} \},
\]
\[
I_2 = \{
(\jmath_{\hat{m}+1},\ldots, \jmath_{m}):  \, 1 \leq \imath_j \leq n_j
\,\mbox{ for }\, j= \hat{m}+1, \ldots, m  \}.
\]
For $ \mF \in \cpx^{n_1,\ldots,n_m}$, its Catalecticant matrix
(here we consider the most square one, cf.~\cite{IaKa99}) is defined as
\be \label{Cat:Mat:nsym}
\mbox{Cat}(\mF) := \Big(\mF_{\imath, \jmath}\Big)_{ \imath \in I_1, \jmath \in I_2}.
\ee
Let $\sig_r(n_1,\ldots,n_m)$ be the closure of the set of all possible sums
$\mF_1+\cdots+\mF_r$, where each $\mF_i$ is a rank-$1$ tensor,
in the Zariski topology (cf.~\cite{CLO07}).
It is an irreducible variety of $\cpx^{n_1,\ldots,n_m}$.
%
%
A property {\tt P} is said to be {\it generically} or {\it generally} true on
$\sig_r(n_1,\ldots,n_m)$ if {\tt P} is true in a nonempty Zariski open subset $T$ of it.
For such a property {\tt P}, we say that $u$
is a generic or general point for {\tt P} if $u \in T$.
For $\mF \in \sig_r(n_1,\ldots,n_m)$, it is possible that $\rank(\mF)>r$.
The {\it border rank} of $\mF$ is defined and denoted as
\be \label{bord:rk:nsym}
\rank_B(\mc{F}) = \min \left\{r: \, \mF \in \sig_r(n_1,\ldots,n_m)  \right\}.
\ee
For all $\mF \in \cpx^{n_1,\ldots, n_m}$, it holds that
\be  \label{rk:Cat<=B:nsym}
\rank\,\mbox{Cat}(\mF) \leq \rank_B(\mF) \leq  \rank(\mF).
\ee
Interestingly, the above three ranks are equal if
$\mF$ is a general point of $\sig_r(n_1,\ldots,n_m)$ and $r$ is small.

\begin{lem}  \label{rank=Cat:nsym}
Let $s$ be the smaller size of $\mbox{Cat}(\mF)$.
For all $r \leq s$, if $\mF$ is a generic point of
$\sig_r(n_1,\ldots,n_m)$, then
\be \label{rkS=Cat:nsym}
 \rank\,(\mF)= \rank\,\mbox{Cat}(\mF) = r.
\ee

\end{lem}
\begin{proof}
Let $\phi_1,\ldots, \phi_k$ be the $r\times r$ minors of the matrix
\[
\mbox{Cat} \Big(  \sum_{i=1}^r
(x^{i,1}) \otimes \cdots  \otimes (x^{i,m}) \Big).
\]
They are homogeneous polynomials in
$x^{i,j}$ ($i=1,\ldots,r$, $j=1,\ldots,m$).
Let $x$ denote the tuple $\big( x^{1,1}, x^{1,2}, \ldots, x^{r,m} \big)$.
Define the projective variety in $\P^{r(n_1+\cdots+n_m)-1}$
\[
Z= \{x : \, \phi_1(x) =\cdots =\phi_k(x) = 0 \}.
\]
Then $Y:=\P^{r(n_1+\cdots+n_m)-1}\backslash Z$
is a Zariski open subset of full dimension (cf.~\cite{CLO07}.)
Consider the polynomial mapping
$\pi:\, Y \to \sig_r(n_1,\ldots,n_m),$
\[
\big( x^{1,1}, x^{1,2}, \ldots, x^{r,m} \big) \mapsto \sum_{i=1}^r
(x^{i,1}) \otimes \cdots  \otimes (x^{i,m}) .
\]
The image $\pi(Y)$ is dense in the irreducible variety $\sig_r(n_1,\ldots,n_m)$.
So, $\pi(Y)$ contains a Zariski open subset, say, $\mathscr{Y}$,
of $\sig_r(n_1,\ldots,n_m)$
(cf.~\cite[Theorem~6,\S5,Chap.I]{Sha:BAG1}).
For each $\mF \in \mathscr{Y}$, there exists
$u\in Y$ such that $\mF= \pi(u)$.
Because $u \not\in Z$, at least one of $\phi_1(u),\ldots,\phi_k(u)$
is nonzero, and hence $\rank\,\mbox{Cat}(\mF) \geq r$.
By \reff{rk:Cat<=B:nsym}, we know that \reff{rkS=Cat:nsym}
is true for all $\mF \in \mathscr{Y}$ because $\rank\,(\mF) \leq r$.
Since $\mathscr{Y}$ is a Zariski open subset of
$\sig_r(n_1,\ldots,n_m)$, the lemma is proved.
\end{proof}

By Lemma~\ref{rank=Cat:nsym},
if $r \leq s$, then for general $\mF \in \sig_r(n_1,\ldots,n_m)$
we have $\rank(\mF)= \rank \, \mbox{Cat}(\mF)=r$.
So, we can use $\rank \, \mbox{Cat}(\mF)$
to estimate $\rank(\mF)$ in practice when $\rank(\mF) \leq s$.
However, for general $\mF \in \cpx^{n_1 \times \cdots \times n_m}$
such that $\rank \, \mbox{Cat}(\mF) = r$, we cannot conclude that
$\mF \in \sig_r(n_1,\ldots,n_m)$.

\subsection{Symmetric tensors and symmetric ranks}
\label{sbsc:symtsr}

Recall that an $m$-tensor $\mF \in \cpx^{n \times \cdots \times n}$
is symmetric if and only if each $\mF_{i_1 \ldots i_m}$ is invariant
under permutations of $(i_1, \ldots, i_m)$.
The space of all such symmetric tensors is denoted as $\mt{S}^m(\cpx^n)$.
In applications, one would naturally write symmetric tensors as
sums of rank-$1$ symmetric tensors.
For $ \mF \in \mt{S}^m(\cpx^n)$, its {\it symmetric rank} is defined as
\be \label{df:rank:sym}
\rank_S\,(\mF) = \min\{ r: \,
(u_1)^{\otimes m}+\cdots+(u_r)^{\otimes m} = \mF \}.
\ee
Clearly, $\rank(\mF) \leq \rank_S(\mF)$.
Let $\sig_r^{m,n}$ be the closure of all possible sums
$(u_1)^{\otimes m}+\cdots+(u_r)^{\otimes m}$,
in the Zariski topology.
The set $\sig_r^{m,n}$ is an irreducible variety.
%
%
Like for $\sig_r(n_1,\ldots,n_m)$,
we say that a property {\tt P} is {\it generically} or {\it generally} true on
$\sig_r^{m,n}$ if it is true in a nonempty Zariski open subset $T$ of $\sig_r^{m,n}$.
For such a property {\tt P}, we say that $u$
is a generic or general point for {\tt P} if $u \in T$.
%
%
%

For $\mF \in \mt{S}^m(\cpx^n)$,
its {\it symmetric border rank} is defined as
\be \label{df:bordrank}
\rank_{SB}(\mc{F}) = \min \left\{r: \, \mF \in \sig_r^{m,n}  \right\}.
\ee
For symmetric tensors, the Catalecticant matrix
defined as in \reff{Cat:Mat:nsym} has repeating rows and columns.
To avoid repeating, for a symmetric tensor $\mF$, we define
its Catalecticant matrix as (cf.~\cite{IaKa99})
\be \label{df:CatMat:F}
\mbox{Cat}(\mF) := (\mF_{\af+\bt})_{|\af|\leq m_1, |\bt|\leq m_2},
\ee
where $m_1 = \lfloor \frac{m}{2} \rfloor$ and $m_2 = \lceil \frac{m}{2} \rceil$.
In \reff{df:CatMat:F}, $\mF$ is indexed by monomial powers
in $n-1$ variables and of degrees $\leq m$.
We refer to \S\ref{sbsc:gr:sym} (also see \cite{GPSTD}) for such indexing.
It always holds that (cf.~\cite{GPSTD})
\be  \label{rk:Cat<=B<=S}
\rank\,\mbox{Cat}(\mF) \leq  \rank_{SB}(\mF) \leq  \rank_S(\mF).
\ee
%
%
A similar version of Lemma~\ref{rank=Cat:nsym}
holds for symmetric tensors.

\begin{lem} \label{pro:rk:Cat=S}
Let $s = \min\{ \binom{n+m_1-1}{m_1}, \binom{n+m_2-1}{m_2} \}$.
For all $r \leq s$,
if $\mF$ is a generic point of $\sig_r^{m,n}$ then
\be \label{r<s:rkS=Cat}
 \rank_S(\mF) = \rank\,\mbox{Cat}(\mF) = r.
\ee
\end{lem}
\begin{proof}
It can be proved in the same way as for Lemma~\ref{rank=Cat:nsym}.
Let $\phi_1,\ldots, \phi_k$ be the $r\times r$ minors of the matrix
\[
\mbox{Cat} \big((x^1)^{\otimes m}+\cdots +(x^r)^{\otimes m} \big),
\]
which has dimension $s \times s$.
Denote $x := (x^1,\ldots,x^r)$, with each $x^i \in \cpx^n$, then
$\phi_1,\ldots, \phi_k$ are homogeneous polynomials in $x$.
%
%
Define the mapping:
\[
\pi:\, Y \to \sig_r^{m,n}, \quad (x^1,\ldots,x^r) \mapsto
(x^1)^{\otimes m}+\cdots +(x^r)^{\otimes m}.
\]
%
%
The rest of the proof is same as for Lemma~\ref{rank=Cat:nsym}.
\end{proof}

Lemma~\ref{pro:rk:Cat=S} implies that
if $r \leq s$, then for general $\mF \in \sig_r^{m,n}$
we have $\rank_S(\mF)= \rank \, \mbox{Cat}(\mF) = r$.
So, in practice, $\rank_S(\mF)$ can be estimated by
$\rank \, \mbox{Cat}(\mF)$ when $\rank_S(\mF) \leq s$.
However, for general $\mF \in \mt{S}^m( \cpx^n )$
such that $\rank \, \mbox{Cat}(\mF)=r$, we may not have
$\mF \in \sig_r^{m,n}$.

\section{Generating polynomials for tensors}
\label{sc:grdef}
\setcounter{equation}{0}

This section introduces generating polynomials
for symmetric and nonsymmetric tensors, and studies their properties.

\subsection{Symmetric tensors}
\label{sbsc:gr:sym}

We consider symmetric tensors in $\mt{S}^m(\cpx^n)$.
For convenience, denote
\[
\on \, := \, n-1, \quad
\N_m^{\on} := \{ \af  \in \N^{\on} : |\af|   \leq m \}.
\]
Let $\cpx[x]:=\cpx[x_1,\ldots,x_{\on}]$ be the ring of complex polynomials in
\[
x := (x_1,\ldots, x_{\on}).
\]
For $\af = (\af_1,\ldots,\af_{\on}) \in \N^{\on}$, denote
$
|\af| := \af_1 + \cdots + \af_{\on}
$
and
$
x^\af := x_1^{\af_1} \cdots x_{\on}^{\af_{\on}}.
$
Each $\mc{F} \in \mt{S}^m(\cpx^{n})$ is indexed by the tuple $(i_1,\ldots, i_m)$.
As introduced in \cite{GPSTD}, we can equivalently
index $\mF$ by monomials $x^\af \in \N_m^{\on}$ as (denote $x_0:=1$)
\[
\mc{F}_{x^\af} := \mc{F}_{i_1, \ldots, i_m}
\quad \mbox{ whenever} \quad
x^\af = x_{i_1-1} \cdots x_{i_m-1}.
\]
For convenience, we also denote
\[
\mc{F}_{\af} \, := \, \mc{F}_{x^\af}.
\]
The dual space of $\mt{S}^m(\cpx^{n})$ is $\cpx[x]_m$,
the space of complex polynomials with degrees $\leq m$.
Each $\mc{F} \in \mt{S}^m(\cpx^{n})$ defines a linear function on $\cpx[x]_m$ as
\be \label{op:scrL:F}
\Big \langle \sum_{\af\in\N_m^{\on} } p_\af x^\af , \mc{F}
\Big \rangle   := \sum_{\af\in\N_m^{\on} } p_\af \mc{F}_\af.
\ee
%
%
For $p\in \cpx[x]_m$ and $\mF \in \mt{S}^m(\cpx^n)$,
$p$ is called a {\em generating polynomial} for $\mF$ if
\be \label{df:GRpq}
\langle p \cdot x^\bt, \mc{F} \rangle  = 0 \quad
\forall\, \bt \in \N_{m-\deg(p)}^{\bn}.
\ee

Generating polynomials are useful for
computing symmetric tensor decompositions (cf.~\cite{GPSTD}).
To compute low rank symmetric tensor approximations,
we need to use generating polynomials for rank-$r$ tensors.
Denote the monomial sets
\be \label{monls:grlex}
\mathbb{B}_0 := \big\{
\underbrace{
1, \, x_1, \, \ldots, \, x_{\bn}, \, x_1^2, \, x_1x_2, \, \ldots
}_{\mbox{ first $r$ monomials }}
\big\},
\ee
\be \label{mscrB12}
\mathbb{B}_1 := \big( \mathbb{B}_0 \cup x_1 \mathbb{B}_0
\cup \cdots \cup x_{\bn} \mathbb{B}_0)
\backslash \mathbb{B}_0.
\ee
For convenience, by $\bt \in \mathbb{B}_0$ (resp., $\af \in \mathbb{B}_1)$,
we mean that $x^\bt \in \mathbb{B}_0$ (resp., $x^\af \in \mathbb{B}_1)$.
Let $\cpx^{ \mathbb{B}_0 \times \mathbb{B}_1 }$ be the space of all complex matrices
indexed by $ (\bt, \af)  \in \mathbb{B}_0 \times \mathbb{B}_1$.
For $\af \in \mathbb{B}_1$ and $G \in \cpx^{ \mathbb{B}_0 \times \mathbb{B}_1 }$,
denote the polynomial in $x$
\be \label{vphi:W-af}
\varphi[G, \af] := \sum_{ \bt  \in \mathbb{B}_0 }
G(\bt,\af)  x^\bt - x^\af.
\ee
If each $\varphi[G, \af]$ ($\af \in \mathbb{B}_1$)
is a generating polynomial for $\mF$,
then $G$ is called a {\it generating matrix} for $\mF$.
Generating matrices as above exist for generic tensors of rank $r$ (cf.~\cite{GPSTD}).

For each $i=1,\ldots, \on$, define the matrix
$M_{x_i}(G) \in \cpx^{\mathbb{B}_0 \times \mathbb{B}_0}$ as:
\be \label{df:Mxi(W)}
M_{x_i}(G)_{\mu, \nu} =
\bca
1  &  \text{ if } x_i \cdot x^\nu \in \mathbb{B}_0, \, \mu = \nu + e_i, \\
0  &  \text{ if } x_i \cdot x^\nu \in \mathbb{B}_0, \, \mu \ne \nu + e_i, \\
G(\mu, \nu+e_i)  &  \text{ if } x_i \cdot x^\nu \in \mathbb{B}_1.
\eca
\ee
The polynomials $\varphi[G, \af]$
have $r$ common zeros (counting multiplicities) if and only if the matrices
$M_{x_1}(G)$, $\ldots$, $M_{x_{\on}}(G)$ commute (cf. \cite{GPSTD}).

\subsection{Nonsymmetric tensors}
\label{sc:GR:nsym}

To define generating polynomials for
nonsymmetric tensors in $\cpx^{n_1 \times \cdots \times n_m}$,
we need a different indexing by multi-linear monomials.

\subsubsection{Definition}

For each $j=1,\ldots,m$, denote
\be \label{barnj:xj}
\on_j := n_j - 1, \quad
\bfx_j \, := \, (x_{j,1}, \ldots, x_{j,\on_j}).
\ee
Let
$
\bfx \, := \, (\bfx_1, \ldots, \bfx_m),
$
the vector of indeterminants
\[
x_{1,1}, \ldots, x_{1,\on_1}, \, \quad x_{2,1}, \ldots, x_{2,\on_2}, \, \quad
\, \ldots, \, \quad  x_{m,1}, \ldots, x_{m,\on_m}.
\]
For convenience, denote
\[
x_{1,0} := 1, \,  x_{2,0} := 1, \, \ldots, \, x_{m,0} := 1,
\]
\be \label{mfrak:mM}
\mathbb{M}:= \Big\{ x_{1,i_1} \cdots x_{m,i_m} \mid
0\leq i_j \leq \on_j, \, 1\leq j \leq m
\Big\}, \quad
\mc{M} = \spn \{ \mathbb{M} \}.
\ee
Each monomial in $\mathbb{M}$ is linear in all $\bfx_j$.
For $J \subseteq \{1,\ldots, m\}$, denote
\[
J^c := \{1,\ldots,m\} \backslash J,
\]
\be \label{mfrak:mJ}
\mathbb{M}_J:= \Big\{ x_{1,i_1} \cdots x_{m,i_m} \mid
i_j = 0 \, \forall \, j \in J^c
\Big\}, \quad
\mc{M}_J:= \spn\{ \mathbb{M}_J \}.
\ee
Each $\mc{F} \in \cpx^{n_1 \times \cdots \times n_m}$
is indexed by $(i_1+1,\ldots,i_m+1)$ with
\[
0\leq i_1 \leq \on_1, \, \ldots,\, 0\leq i_m \leq \on_m.
\]
Since $(i_1+1,\ldots,i_m+1)$ is uniquely determined by
the multi-linear monomial $x_{1,i_1} \cdots x_{m,i_m} \in \mathbb{M}$,
we can equivalently index $\mc{F}$ as
\[
\mc{F}_{ x_{1,i_1} \cdots x_{m,i_m} } :=  \mc{F}_{i_1+1,\ldots,i_m+1}.
\]
The tensors in $\cpx^{n_1 \times \cdots \times n_m}$
can be equivalently indexed by multi-linear monomials in $\mathbb{M}$.
For $\mc{F} \in \cpx^{n_1 \times \cdots \times n_m}$,
define the operation on polynomials in $\mc{M}$ as
\be \label{<p,F>:nsym}
\big\langle \sum_{\mu \in \mathbb{M} } c_\mu \mu, \mc{F} \big \rangle
\, := \, \sum_{\mu \in \mathbb{M} }  c_\mu \mc{F}_\mu.
\ee
In the above, each $c_\mu$ is a complex scalar.

\begin{defi}  \label{grdef:nsym}
For $J \subseteq \{1,\ldots,m\}$ and
$\mc{F} \in \cpx^{n_1 \times \cdots \times n_m}$,
we call $p \in \mc{M}_J$ a
{\it generating polynomial} for $\mc{F}$ if
\be \label{df:gr:nsym}
\langle pq, \mc{F}   \rangle = 0 \quad \forall\, q \in \mathbb{M}_{J^c}.
\ee
\end{defi}

\begin{exm}
\label{exm:gprk1:ns}
Consider the rank-$1$ tensor $\mc{F} = \lmd ( a \otimes b \otimes c)$ with
$\lmd \in \cpx$,
\[
a = (1, a_1, \ldots, a_{\on_1} ), \quad b = (1, b_1, \ldots, b_{\on_2} ),
\quad c = (1, c_1, \ldots, c_{\on_3} ).
\]
For all $i,j,k\geq 1$, the polynomials
\[
f_i :=x_{1,i} - a_i, \quad g_j :=x_{2,j} - b_j,  h_k :=x_{3,k} - c_k
\]
are all generating polynomials for $\mc{F}$. This is because
\[
\langle (x_{1,i} - a_i)x_{2,j}x_{3,k}, \mc{F}   \rangle =
(a_i-a_i) b_j c_k = 0 \quad \forall \, j,\, k \geq 1,
\]
\[
\langle (x_{2,j} - b_j)x_{1,i}x_{3,k}, \mc{F}   \rangle =
(b_j - b_j) a_i c_k = 0 \quad \forall \, i, \, k \geq 1,
\]
\[
\langle (x_{3,k} - c_k)x_{1,i}x_{2,j}, \mc{F}   \rangle =
(c_k - c_k) a_i b_j = 0 \quad \forall \, i, \, j \geq 1.
\]
\end{exm}

\begin{exm}
Consider the rank-$2$ tensor
\[
\mc{F} = \lmd (a \otimes b \otimes c )  + \mu ( d \otimes e \otimes f)
\]
with $\lmd, \mu \in \cpx$, $a,b,c$ as in Example~\ref{exm:gprk1:ns}, and
\[
d = (1, d_1, \ldots, d_{\on_1} ), \quad e = (1, e_1, \ldots, e_{\on_2} ),
\quad f = (1, f_1, \ldots, f_{\on_3} ).
\]
One can check that for all $i,j,k\geq 1$, the products
\[
(x_{1,i} - a_i)(x_{2,j} - e_j), \quad (x_{1,i} - a_i)(x_{3,k} - f_k),
\]
\[
(x_{2,j} - b_j)(x_{1,i}-d_i), \quad (x_{2,j} - e_j)(x_{3,k}-f_k),
\]
\[
(x_{3,k}-c_k)(x_{1,i}-d_i), \quad (x_{3,k}-c_k)(x_{2,j}-ej)
\]
are generating polynomials for $\mF$. For instance, one can check that
\[
\langle (x_{1,i} - a_i)(x_{2,j} - e_j) x_{3,k}, \mF  \rangle =
\lmd \langle (x_{1,i} - a_i)(x_{2,j} - e_j) x_{3,k}, a \otimes b \otimes c  \rangle  +
\]
\[
+ \mu
\langle (x_{1,i} - a_i)(x_{2,j} - e_j) x_{3,k}, d \otimes e \otimes f  \rangle = 0.
\]
\end{exm}

\subsubsection{Properties}

For convenience, denote
\be \label{nota:r1:J}
\left\{ \baray{c}
r_1 \, := \,r - 1,  \quad  \, [r_1]:=\{0,1,\ldots,r_1\}, \\
J := \Big\{ (i,j,k):  \,\,
0\leq i \leq r_1, \quad 2 \leq j \leq m, \quad 1 \leq k \leq \on_j
\Big\}.
\earay \right.
\ee
For $G \in \cpx^{[r_1] \times J}$ and
$\tau =(i,j,k)$, denote the polynomial
\be \label{vphi:ijk:w}
\varphi[G,\tau](\bfx) :=
\sum_{\ell=0}^{r_1} G(\ell, \tau) x_{1,\ell}  - x_{1,i} x_{j,k}.
\ee
The rows of $G$ are indexed by $\ell = 0, 1, \ldots, r_1$
and the columns are indexed by $\tau \in J$.
If $\varphi[G,\tau]$ is a generating polynomial for $\mc{F}$, then
\[
\langle \varphi[G,\tau] \cdot \mu,  \mc{F}  \rangle  = 0,
\quad \forall \, \mu \in \mathbb{M}_{ \{1,j\}^c }.
\]
The above gives linear equations ($\tau =(i,j,k)$)
\be \label{F:ijk:w=0}
\sum_{\ell=0}^{r-1} \mc{F}_{x_{1,\ell} \cdot \mu }
G(\ell, \tau) = \mc{F}_{x_{1,i} \cdot x_{j,k} \cdot \mu }
\quad \forall \, \mu \in \mathbb{M}_{ \{1,j\}^c }.
\ee

\begin{defi} \label{df:GM:nsym}
For $\mF \in \cpx^{n_1 \times \cdots \times n_m}$ and $G \in \cpx^{[r_1] \times J}$,
if \reff{F:ijk:w=0} is satisfied for all $\tau \in J$,
then $G$ is called a {\it generating matrix} for $\mF$.
\end{defi}

Generating matrices as above exist for general rank-$r$ tensors. Suppose
\be \label{dc:F=sum:u1:ur}
\mc{F} =  \sum_{s=1}^r
u^{s,1} \otimes u^{s,2} \otimes \cdots \otimes u^{s,m}.
\ee
For convenience, we index vectors in $\cpx^{n_j}$ by $i=0,1, \ldots, \bn_j$.
Up to a scaling, we generally assume that
\be \label{u:ell:j:0=1}
(u^{s,2})_0 = \cdots = (u^{s,m})_0 = 1.
\ee
For convenience, denote
\be \label{nota:usj:1:r1}
\big(u^{s, j}\big)_{0:r_1}  :=
\bbm (u^{s,j})_0 \\  (u^{s,j})_1 \\ \vdots \\  (u^{s,j})_{r_1}  \ebm, \quad
u^{(s)}\,  := \, (u^{s,1}, \ldots,  u^{s,m} ),
\ee
\be \label{nota:U:U1}
U = \Big( [u^{(1)}]_{ \mathbb{M} } \quad \cdots \quad [u^{(r)}]_{ \mathbb{M} }
\Big), \quad
U_1 = \Big( \big(u^{1, 1}\big)_{0:r_1} \quad \cdots \quad \big(u^{r, 1}\big)_{0:r_1}
\Big).
\ee
In the above, $[u^{(s)}]_{ \mathbb{M} }$ denotes the vector
that is the evaluation of the monomials in $\mathbb{M}$
at the tuple $u^{(s)}=(u^{s,1}, \ldots, u^{s,r})$.
The rows of $U$ are index by monomials in $\mathbb{M}$.
Note that $U_1$ is a $r\times r$ square matrix,
consisting of the first $r$ rows of $U$. Clearly,
if $\rank\, U_1 = r$, then each row of $U$
is a linear combination of the rows of $U_1$.
So, for the ($x_{1,i}x_{j,k}$)-th row,
there exists a vector $w^{i,j,k} \in \cpx^{r}$ such that
\[
U(x_{1,i}x_{j,k}, :) = \sum_{\ell=0}^{r-1}
(w^{i,j,k})_\ell  U(x_{1,\ell}, :).
\]
For each $s = 1,\ldots, m$, \reff{nota:U:U1} and the above imply that
\[
(u^{s,1})_i (u^{s,j})_k = \sum_{\ell=0}^{r-1}
(w^{i,j,k})_{\ell}  (u^{s,1})_{\ell} ,
\]
\[
(u^{(s)})^{\mu}  (u^{s,1})_i (u^{s,j})_k = \sum_{\ell=0}^{r-1}
(w^{i,j,k})_{\ell}    (u^{(s)})^{\mu}   (u^{s,1})_\ell
\quad \forall \, \mu \in \mathbb{M}_{\{1,j\}^c }.
\]
Let $G \in \cpx^{[r_1] \times J}$ be the matrix such that
\be \label{G(tau)=wijk}
G(\ell,\tau) = (w^{i,j,k})_{\ell} \quad
\big(0\leq \ell \leq r_1, \ \tau = (i,j,k) \in J \big).
\ee
Then, one can verify that
\[
\langle \varphi[G,\tau]\cdot \mu,
u^{s,1} \otimes \cdots \otimes u^{s,m} \rangle = 0
\quad \forall \, \mu \in \mathbb{M}_{\{1,j\}^c },
\]
\[
\langle \varphi[G,\tau]\cdot \mu,  \mc{F} \rangle = 0
\quad \forall \, \mu \in \mathbb{M}_{\{1,j\}^c }.
\]
That is, each $\varphi[G,\tau]$ is a generating polynomial for $\mc{F}$.
Hence, we get the following theorem.

\begin{theorem} \label{thm:exist:wijk}
Let $\mc{F}$ be as in \reff{dc:F=sum:u1:ur}-\reff{u:ell:j:0=1},
$U_1$ be as in \reff{nota:U:U1}. Suppose $\rank \, U_1 =r$.
If $G$ is given as in \reff{G(tau)=wijk}, then
$G$ is a generating matrix for $\mF$.
\end{theorem}

For a general rank-$r$ tensor $\mc{F}$,
we typically have $\rank \, U_1 =r$. Otherwise,
it can be satisfied by applying a general linear transformation.
For each pair $(j,k)$, denote
%
%
{\smaller
\be \label{df:Mjk}
M^{j,k}(G) :=
\bbm
G(0, (0,j,k)) & G(1, (0,j,k)) & \cdots & G(r_1, (0,j,k))  \\
G(0, (1,j,k)) & G(1, (1,j,k)) & \cdots & G(r_1, (1,j,k)) \\
 \vdots & \vdots & \ddots  & \vdots \\
G(0, (r_1,j,k)) & G(1, (r_1,j,k)) & \cdots & G(r_1, (r_1,j,k))
\ebm.
\ee
}

\begin{prop} \label{prop:Mjk:eig}
In Theorem~\ref{thm:exist:wijk}, for $G$ as in \reff{G(tau)=wijk}, we have
\be \label{eigsym:Mjk}
M^{j,k}(G) \cdot \big(u^{s,1}\big)_{0:r_1} = ( u^{s,j})_k  \cdot
\big(u^{s,1}\big)_{0:r_1}
\ee
for all $s=1,\ldots,r$, $j=2,\ldots,m$ and $0 \leq k \leq \on_j$.
\end{prop}
\begin{proof}
This can be done by a direct verification.
\end{proof}

By \reff{eigsym:Mjk}, each $( u^{s,j})_k $ ($j=2,\ldots,m$)
is a common eigenvalue of $M^{j,k}(G)$,
with $\big(u^{s,1}\big)_{0:r_1}$ a common eigenvector.
Therefore, if the generating matrix $G$
is known, we can get $u^{s,j}$ from
the common eigenvalues and eigenvectors of $M^{j,k}(G)$.

\section{Low rank approximations for symmetric tensors}
\label{sc:lra:sym}
\setcounter{equation}{0}

Given a symmetric tensor $\mc{F} \in \mt{S}^m(\cpx^n)$ and a rank $r$,
we are looking for $\mc{X} \in \mt{S}^m(\cpx^n)$,
such that its symmetric rank $\rank_S(\mc{X}) \leq r$
and it is close to $\mc{F}$ as much as possible.
We refer to \S\ref{sbsc:symtsr} for symmetric ranks of symmetric tensors.
Clearly, $\rank_S(\mc{X}) \leq r$ if and only if
$\mc{X} = (u_1)^{\otimes m} + \cdots + (u_r)^{\otimes m}$
for vectors $u_1, \ldots, u_r \in \cpx^n$.
Finding such $\mc{X}$ is equivalent to solving the nonlinear least squares problem
\be \label{apx:LS:u}
\min_{ u_1, \ldots, u_r \in \cpx^n } \quad
\left \|  (u_1)^{\otimes m} + \cdots
+ (u_r)^{\otimes m} - \mF \right \|^2.
\ee
Throughout the paper, for symmetric tensors,
by rank-$r$ approximation we mean that
the approximating tensor $\mc{X}$ is symmetric and
$\rank_S\,(\mc{X}) \leq r$.

It is typically hard to solve \reff{apx:LS:u} directly.
For the basic case $r=1$ (i.e., best rank-$1$ approximations),
the problem is already NP-hard (cf.~\cite{HiLi13}).
We propose to compute rank-$r$ approximations by using generating polynomials
(cf.~\S\ref{sbsc:gr:sym}).
For convenience,  we index vectors in $\cpx^n$
by $i =0, 1,\ldots, \bn$, where \[ \bn := n-1. \]
For $u \in \cpx^n $ with $(u)_0 \ne 0$,
we can write it equivalently as
\[
u = \lmd^{1/m} v, \quad \lmd \in \cpx, \quad (v)_0 = 1, \, v \in \cpx^n.
\]
Then, $u^{\otimes m}  = \lmd  v^{\otimes m}$.
So, \reff{apx:LS:u} can be reformulated as
\be \label{apx:LS:lmd*v}
\min_{
\substack{v_1, \ldots, v_r \in \cpx^n, \lmd_1, \ldots, \lmd_r \in \cpx \\
(v_1)_0 = \cdots = (v_r)_0 = 1  }
} \quad
\| \lmd_1 (v_1)^{\otimes m} + \cdots + \lmd_r (v_r)^{\otimes m}
- \mF \|^2.
\ee
In this section, we propose to solve \reff{apx:LS:lmd*v} in three major stages:
\bit

\item Find an approximate generating matrix $G$ for $\mc{F}$.

\item Construct $v_1,\ldots,v_r$ from the polynomials $\varphi[G,\af]$
(see \reff{vphi:W-af}).

\item Determine $\lmd_1,\ldots,\lmd_r$
by solving the resulting least squares.

\eit

\subsection{An algorithm for rank-$r$ approximation}
\label{sc:lrap:alg}

Let $\mathbb{B}_0,\mathbb{B}_1$ be the monomial sets as in
\reff{monls:grlex}-\reff{mscrB12}.
By $\af \in \mathbb{B}_1$ we mean that $x^\af \in \mathbb{B}_1$.
We index symmetric tensors by monomials, or equivalently
by monomial powers (cf.~\S\ref{sbsc:gr:sym}).
For a matrix $G \in \cpx^{\mathbb{B}_0 \times \mathbb{B}_1}$
and $\af \in \mathbb{B}_1$, the polynomial $\varphi[G,\af]$
as in \reff{vphi:W-af} is a generating polynomial for $\mc{F}$ if
\be \label{<gmWaf:F>=0}
\Big \langle x^\gm \varphi[G, \af], \mc{F} \Big\rangle = 0 \quad
\big( \forall\, \gm \in \N_{m-|\af|}^{\on} \big).
\ee
(See \reff{op:scrL:F} for the operation $\langle \cdot, \cdot \rangle$.)
A matrix $G$ satisfying \reff{<gmWaf:F>=0} may, or may not, exist.
But we can always find an approximate one
by solving the linear least squares
\be \label{ls:rc:Waf}
\min_{ G \in \cpx^{\mathbb{B}_0 \times \mathbb{B}_1} } \quad
\sum_{ \af \in \mathbb{B}_1 }  \sum_{ \gm \in  \N_{m-|\af|}^{\on} }
\Big| \big
\langle x^\gm \varphi[G,\af], \mc{F}
\big\rangle \Big|^2.
\ee
Denote the matrix $A[\mc{F},\af]$ and the vector $b[\mc{F},\af]$ as
\be \label{df:Ab[F,af]}
\left\{ \baray{lcl}
A[\mc{F},\af]_{\gm , \bt} &=& \mc{F}_{\bt+\gm}, \quad
\forall \,  (\gm, \bt) \in \N_{m-|\af|}^{\on} \times \mathbb{B}_0, \\
b[\mc{F},\af]_{\gm} &=& \mc{F}_{\af+\gm}, \quad
\forall \,  \gm \in \N_{m-|\af|}^{\on}.
\earay\right.
\ee
The matrix $G$ is indexed by $(\bt, \af) \in \mathbb{B}_0 \times \mathbb{B}_1$.
Denote by $G(:,\af)$ the $\af$-th column of $G$.
So, \reff{ls:rc:Waf} can be reformulated as
\be \label{ls:Aw=b:af}
\min_{ G \in \cpx^{\mathbb{B}_0 \times \mathbb{B}_1} } \quad
\sum_{ \af \in \mathbb{B}_1 }
\Big\| A[\mc{F},\af] \, G(:,\af) - b[\mc{F},\af] \Big\|_2^2.
\ee
Let $G^{ls}$ be an optimal solution to \reff{ls:Aw=b:af}.
%
%
When $\mF$ is close to $\sig_r^{m,n}$,
the set of tensors in $\mt{S}^m(\cpx^n)$ whose symmetric
border ranks $\leq r$ (cf.~\S\ref{sbsc:symtsr}), the polynomial system
\be \label{vpiWLS(x)=0}
\varphi[G^{ls}, \af](x) = 0 \quad (\af \in \mathbb{B}_1)
\ee
is expected to approximately have $r$ solutions.
In practice, we can find them
by using the method in \cite{CGT97} (also see \cite{GPSTD}).
Choose generic $\xi_1>0,\ldots, \xi_{\bn}>0$ and scale them as
$\xi_1 +\cdots + \xi_{\bn} = 1$. Let
\be \label{M(xiWls)}
M(\xi,G^{ls}) := \xi_1 M_{x_1}(G^{ls}) +
\cdots +  \xi_{\on} M_{x_{\on}}(G^{ls}),
\ee
where $M_{x_i}(G)$ is defined as in \reff{df:Mxi(W)}.
Then, compute the Schur Decomposition
\be  \label{Schur:MWls}
Q^* M(\xi,G^{ls}) Q \, = \, T,
\ee
where $Q=[q_1 \, \ldots \, q_r ]$ is unitary and
$T$ is upper triangular.
(The superscript $^*$ denotes the conjugate transpose.)
For $i=1,\ldots,r$, let
\be \label{hatv:ls:aprx}
v_i^{ls} :=
\left(1, \,\, q_i^* M_{x_1}(G^{ls})q_i, \, \, \ldots, \, \,
q_i^* M_{x_{\on}}(G^{ls}) q_i \right).
\ee
Then, we solve the linear least squares
\be  \label{LS:vls*lmd=F}
\min_{ \substack{ (\lmd_1,\ldots,\lmd_r) \in \cpx^r } } \quad
\| \lmd_1 (v_1^{ls}) ^{\otimes m} + \cdots +
\lmd_r (v_r^{ls})^{\otimes m} -\mc{F} \|^2.
\ee
Let $\lmd^{ls}:=(\lmd_1^{ls}, \ldots, \lmd_r^{ls})$
be an optimal solution to \reff{LS:vls*lmd=F}.
For $i=1,\ldots,r$, let $u_i^{ls} := \sqrt[m]{\lmd_i^{ls}} v_i^{ls}$.
Then, we construct the tensor
\be \label{Xls:sym}
\mc{X}^{gp} := (u_1^{ls})^{\otimes m}+ \cdots + (u_r^{ls})^{\otimes m}.
\ee

The low rank tensor $\mc{X}^{gp}$ as in \reff{Xls:sym}
is expected to be a good rank-$r$ approximation,
when $\mF$ is close to a rank-$r$ tensor.
Its quality is analyzed in Theorem~\ref{thm:lrkapx:err}.
Our main conclusion is that if $\mF$ is close enough to a rank-$r$ tensor,
then $\mc{X}^{gp}$ is a good enough rank-$r$ approximation.
The tensor $\mc{X}^{gp}$ is not mathematically guaranteed to be a best
rank-$r$ approximation. However, we can always improve it
by using classical nonlinear least squares methods.
This often can be done efficiently, because
$\mc{X}^{gp}$ is typically a good approximation.

\medskip

The above leads to the following algorithm.

\begin{alg} \label{alg:tsr:aprox}
Symmetric rank-$r$ approximations for symmetric tensors. \\
For given $\mc{F} \in \mt{S}^m(\cpx^n)$ and given $r$,
do the following:
\bit

\item [Step 1] Solve \reff{ls:Aw=b:af} for a
least square solution $G^{ls}$.

\item [Step 2] Choose generic $\xi_i>0$
(scale as $\sum_i \xi_i= 1$), and
formulate $M(\xi,G^{ls})$ as in \reff{M(xiWls)}.
Compute the Schur Decomposition \reff{Schur:MWls}.

\item [Step 3] Compute $v_1^{ls},\ldots, v_r^{ls}$ as in
\reff{hatv:ls:aprx}, and solve \reff{LS:vls*lmd=F} for
$(\lmd_1^{ls}, \ldots, \lmd_r^{ls})$.

\item [Step 4] Construct the tensor $\mc{X}^{gp}$ as in \reff{Xls:sym}.

\item [Step 5] Solve \reff{apx:LS:u} for
an improved solution $(u_1^{opt},\ldots, u_r^{opt})$,
with the starting point $(u_1^{ls},\ldots, u_r^{ls})$.
Output the tensor
\be \label{Xgr:sym}
\mc{X}^{opt} := (u_1^{opt})^{\otimes m}+ \cdots + (u_r^{opt})^{\otimes m}.
\ee

\eit

\end{alg}

\bigskip

\begin{remark} \label{est:rk:cat}
In Algorithm~\ref{alg:tsr:aprox}, the approximating rank $r$ need to be given.
In practice, such $r$ is often not known in advance.
However, we can estimate it by using the Catalecticant matrix $\mbox{Cat}(\mF)$.
Let $s$ be the smaller size of $\mbox{Cat}(\mF)$.
If $\mF = \mc{X} + \mc{E}$, then
\[
\mbox{Cat}(\mF) = \mbox{Cat}(\mc{X}) + \mbox{Cat}(\mc{E}).
\]
If $\mc{X}$ is a generic point of $\sig_r^{m,n}$ and $r \leq s$,
then $\rank_S (\mc{X}) = \rank \mbox{Cat}(\mc{X}) =r$,
by Lemma~\ref{pro:rk:Cat=S}. For such case,
$\rank_S (\mc{X})$ can be estimated by $\rank \mbox{Cat}(\mc{X})$.
When $\mc{E}$ is small, $\rank\,\mbox{Cat}(\mc{X})$
can be estimated by $\rank \,\mbox{Cat}(\mF)$.
In practice, we can do as follows:
compute the singular values of $\mbox{Cat}(\mF)$, say,
$\eta_1 \geq \eta_2 \geq \cdots $.
If $\eta_r \gg \eta_{r+1}$, then
such $r$ is a good estimate for $\rank_S (\mc{X})$.
%
%
\end{remark}

When $r=1$, we can get an explicit formula for $\mc{X}^{gp}$ in Algorithm~\ref{alg:tsr:aprox}.
For such case, $\mathbb{B}_0=\{1\}$, $\mathbb{B}_1= \{x_1, \ldots, x_n\}$,
and each $A[\mc{F}, x_i]$ is the vector
\[
a_i \, := \, (\mc{F}_\gm)_{ \gm \in \N_{m-1}^{\on} }.
\]
For $i=1,\ldots,\on$, let
\be \label{r=1:fml:taui}
\theta_i \, := \, a_i^*b[\mc{F},x_i] / a_i^*a_i.
\ee
Then, $v^{ls} \, = \, (1, \theta_1, \ldots, \theta_{\on})$ and (let $\theta_0=1$)
\be \label{r=1:fml:lmdls}
\lmd^{ls} =
\Big( \sum_{ 1 \leq i_1, \ldots, i_m \leq n}
\mc{F}_{i_1,\ldots,i_m}  \theta_{i_1-1}^* \cdots \theta_{i_m-1}^*
\Big)/ \| (v^{ls})^{\otimes m} \|_2^2.
\ee
So, the rank-$1$ tensor $\mc{X}^{gp} = \lmd^{ls} (v^{ls})^{\otimes m}$
can be computed by a closed formula.

\subsection{Performance analysis}
\label{sbsc:symerr}

We analyze the performance of Algorithm~\ref{alg:tsr:aprox}.
Suppose the tensor
\be \label{dcmp:Xbs}
\mc{X}^{bs} := (u_1^{bs})^{\otimes m} + \cdots + (u_r^{bs})^{\otimes m}
\ee
is a best, or an almost best, rank-$r$ approximation for $\mc{F}$. Let
\be \label{F=Xbs+E}
\mc{E} =  \mF - \mc{X}^{bs}, \quad \| \mc{E} \| = \eps.
\ee
Let $A[\mc{F},\af],\,b[\mc{F},\af]$ be as in \reff{df:Ab[F,af]},
then for each $\af \in \mathbb{B}_1$
\be \label{Ab:F=Xls+E}
\left\{ \baray{rcr}
A[\mc{F},\af] &=& A[\mc{X}^{bs},\af] + A[\mc{E},\af], \\
b[\mc{F},\af] &=& b[\mc{X}^{bs},\af] + b[\mc{E},\af].
\earay \right.
\ee
Recall that we index vectors in $\cpx^n$ by $j=0,1,\ldots,\bn$.
Suppose each $(u_i^{bs})_0 \ne 0$, so we can scale them as ($i=1,\ldots,r$)
\[
u_i^{bs} = (\lmd_i^{bs})^{1/m} \,  v_i^{bs} ,
\quad \lmd_i^{bs} \in \cpx, \quad (v_i^{bs})_0 = 1,  \quad v_i^{bs} \in \cpx^n.
\]
For convenience of notation, denote
\[
\overline{ v_i^{bs} } \, := \,
\bbm (v_i^{bs})_1 & \cdots & (v_i^{bs})_{\bn} \ebm^T.
\]
The tuple $(u_1^{bs},\ldots, u_r^{bs})$ is called {\it scaling-optimal} for $\mF$
if $(\lmd_1^{bs},\ldots, \lmd_r^{bs})$ is an optimizer of
the linear least squares
\be \label{opt-scal:vbs}
\min_{ (\lmd_1, \ldots, \lmd_r) \in \cpx^r } \quad
\| \lmd_1 (v_1^{bs})^{\otimes m} + \cdots +
\lmd_r (v_r^{bs})^{\otimes m} -\mc{F} \|^2.
\ee
It is reasonable to assume that $(u_1^{bs},\ldots, u_r^{bs})$
is scaling-optimal, because otherwise we can
replace it by a scaling-optimal one in the choice of $\mc{X}^{bs}$
which is assumed to be a best, or an almost best,
rank-$r$ approximation for $\mF$.

The quality of the low rank tensors
$\mc{X}^{gp}, \mc{X}^{opt}$ is estimated as follows.
For each $i$, the notation $[\overline{ v_i^{bs} }]_{\mathbb{B}_0}$
denotes the vector of monomials in the set $\mathbb{B}_0$
as in \reff{monls:grlex}, evaluated at the point $\overline{ v_i^{bs} }$.
Denote $\xi = \bbm \xi_1 & \cdots & \xi_{\bn} \ebm^T$.

\begin{theorem} \label{thm:lrkapx:err}
Let $\mc{F}, \mc{X}^{bs}, \mc{E}, u_i^{bs}, v_i^{bs}, \lmd_i^{bs}$ be as above.
Assume that
\bit

\item [i)] the vectors $[\overline{ v_1^{bs} }]_{\mathbb{B}_0},\ldots,
[\overline{ v_r^{bs} }]_{\mathbb{B}_0}$ are linearly independent;

\item [ii)] each matrix $A[\mF,\af] \,(\af \in \mathbb{B}_1)$ has full column rank;

\item [iii)] the tuple $(u_1^{bs},\ldots, u_r^{bs})$ is scaling-optimal for $\mF$;

\item [iv)] the scalars
$\xi^T \overline{ v_1^{bs} },\ldots, \xi^T \overline{ v_r^{bs} }$ are distinct from each other.

\eit
If $\mc{X}^{gp}, \mc{X}^{opt}$ are from Algorithm~\ref{alg:tsr:aprox}
and $\eps = \| \mc{E} \|$ is small enough, then
\be \label{err:F-Xrc:bs}
\| \mc{X}^{bs} - \mc{X}^{gp} \| = O(\eps), \quad
\| \mc{F} - \mc{X}^{opt} \| \leq \| \mc{F} - \mc{X}^{gp} \| = O(\eps),
\ee
where the constants in the above $O(\cdot)$ only depend on $\mF$ and $\xi$.
\end{theorem}
\begin{proof} \,
By the condition i) and Theorem~3.1 of \cite{GPSTD},
there exists a generating matrix
$G^{bs} \in \cpx^{\mathbb{B}_0 \times \mathbb{B}_1}$ for $\mc{X}^{bs}$,
i.e., for all $\af \in \mathbb{B}_1$,
\[
\varphi[G^{bs},\af](\overline{ v_1^{bs} }) = \cdots =
\varphi[G^{bs},\af](\overline{ v_r^{bs} })  = 0,
\]
\[
A[\mc{X}^{bs},\af] \, G^{bs}(:,\af) = b[\mc{X}^{bs},\af].
\]
By \reff{F=Xbs+E}, for all $\af \in \mathbb{B}_1$, we have
\[
\big\| A[\mc{F},\af]-A[\mc{X}^{bs},\af] \big\|_F \leq \eps, \quad
\big\| b[\mc{F},\af]-b[\mc{X}^{bs},\af] \big\|_2 \leq \eps,
\]
\[
G^{ls}(:,\af) = \big(A[\mc{F},\af]\big)^\dag b[\mc{F},\af], \quad
G^{bs}(:,\af) = \big(A[\mc{X}^{bs},\af]\big)^\dag b[\mc{X}^{bs},\af].
\]
(The superscript $^\dag$ denotes the Moore-Penrose pseudoinverse.)
By the condition ii), for $\eps$ small enough, we have
\[
\| G^{ls}(:,\af) - G^{bs}(:,\af) \|_2 = O(\epsilon).
\]
The constant in the above $O(\epsilon)$ only depends on $\mF$. Hence,
$
\| G^{ls} - G^{bs} \|_F = O(\epsilon),
$
\[
\| M(\xi,G^{bs}) - M(\xi,G^{ls}) \|_F
= \| \sum_{i=1}^n \xi_i M_{x_i}(G^{bs}-G^{ls}) \|_F
\]
\[
\leq  \big(\sum_{i} \xi_i \big)\, \max_i \|  M_{x_i}(G^{bs}-G^{ls})  \|_F
\leq \|  G^{bs}-G^{ls}  \|_F = O(\epsilon).
\]
In the above, the last inequality follows from the definition
of $M_{x_i}(G)$ as in \reff{df:Mxi(W)}.
For each $i$, $(v_i^{bs})_1, \ldots, (v_i^{bs})_{\on}$
are respectively the eigenvalues of
\[
M_{x_1}(G^{bs}), \,  \ldots, \, M_{x_{\on}}(G^{bs}),
\]
with a common eigenvector (cf.~\cite[\S2]{GPSTD}).
%
%
The eigenvalues of $M(\xi, G^{bs})$ are
\[
\xi^Tv_1^{bs}, \ldots, \xi^T v_r^{bs}.
\]
By the condition iv), $M(\xi, G^{bs})$ does not have a repeated eigenvalue,
and so is $M(\xi, G^{ls})$ for $\eps>0$ small enough. For such $\eps$,
there exist a unitary matrix $Q_1$ and an upper triangular matrix $T_1$ such that
\be \label{SchurQTerr}
Q_1^* M(\xi, G^{bs}) Q_1 = T_1, \quad
\| Q - Q_1\|_F = O(\eps), \quad \|T-T_1\|_F = O(\eps).
\ee
This can be implied by \cite{KPC94} or \cite[Theorem~4.1]{Sun95}.
The constants in $O(\eps)$ of \reff{SchurQTerr}
depend on $G^{ls}$, and then eventually only depend on $\mF$ and $\xi$.

\medskip

The matrices $M(\xi, G^{bs})$ and $T_1$ have common eigenvalues.
Because $T_1$ is upper triangular and its diagonal entries are all distinct,
there exists an upper triangular nonsingular matrix $R_1$ such that
$
\Lambda_1 \, := \, R_1^{-1} T_1 R_1
$
is diagonal. This results in the eigenvalue decomposition
\[
(Q_1R_1)^{-1} M(\xi, G^{bs}) (Q_1R_1) = \Lambda_1.
\]
Since $\overline{ v_1^{bs} },\ldots, \overline{ v_r^{bs} }$ are pairwisely distinct,
implied by the condition i),
the polynomials $\varphi[G^{bs},\af]$ ($\af \in \mathbb{B}_1$)
do not have repeated zeros. So, the companion matrices
$
M_{x_1}(G^{bs}), \ldots, M_{x_{\on}}(G^{bs})
$
can be simultaneously diagonalized (cf.~\cite[Corollary~2.7]{Stu02}).
There exist a nonsingular matrix $P$
and diagonal matrices $D_1,\ldots,D_n$ such that
\[
P^{-1} M_{x_1}(G^{bs}) P = D_1, \quad \ldots, \quad
P^{-1} M_{x_n}(G^{bs}) P = D_n,
\]
\[
P^{-1} M(\xi, G^{bs}) P = \sum_i \xi_i D_i .
\]
Since $M(\xi, G^{bs})$ does not have a repeated eigenvalue,
each eigenvector is unique, up to a scaling.
So, there exists a diagonal matrix $D_0$ such that
$ Q_1 R_1 = P D_0.$
For each $j$, the matrix
\[
Q_1^* M_{x_j}(G^{bs}) Q_1 = R_1 D_0^{-1}  P^{-1} M_{x_j}(G^{bs}) P D_0 R_1^{-1}
\]
\[
= R_1 D_0^{-1} D_j D_0  R_1^{-1} \,
= \, R_1   D_j  R_1^{-1}  \overset{def}{=} \, S_j
\]
is upper triangular. The diagonals of $D_j$ and $S_j$ are same.
So, for all $i,j$ we have
\[
Q_1(:,i)^*M_{x_j}(G^{bs})Q_1(:,i) =
P^{-1}(i,:)M_{x_j}(G^{bs})P(:,i).
\]
($P^{-1}(i,:)$ denotes the $i$-th row of $P^{-1}$.)
For each $i$, the vector
\[
\bbm
P^{-1}(i,:)M_{x_1}(G^{bs})P(:,i) \\ \vdots \\ P^{-1}(i,:)M_{x_n}(G^{bs})P(:,i))
\ebm
\]
is one of $\overline{ v_1^{bs} },\ldots, \overline{ v_r^{bs} }$.
Up to a permutation of indices, we have
\[
\overline{ v_i^{bs} } =
\bbm
Q_1(:,i)^*M_{x_1}(G^{bs})Q_1(:,i) \\ \vdots
\\ Q_1(:,i)^*M_{ x_{\bar{n}} }(G^{bs})Q_1(:,i))
\ebm.
\]
By \reff{hatv:ls:aprx} and \reff{SchurQTerr}, the above implies that
\be \label{vi:ls=bs+eps}
v_i^{ls} = v_i^{bs} + O(\eps) \, \,( i = 1,\ldots,r).
\ee
The constants in the above $O(\eps)$ eventually only depend on $\mF$ and $\xi$.

In Algorithm~\ref{alg:tsr:aprox},
$(\lmd_1^{ls}, \ldots, \lmd_r^{ls})$ is an optimizer of
\[
\min_{ (\lmd_1, \ldots, \lmd_r) \in \cpx^r }
\| \lmd_1 (v_1^{ls})^{\otimes m} + \cdots +
\lmd_r  (v_r^{ls})^{\otimes m} - \mF \|_2^2,
\]
while $(\lmd_1^{bs}, \ldots, \lmd_r^{bs})$ is an optimizer of
\[
\min_{ (\lmd_1, \ldots, \lmd_r) \in \cpx^r }
\| \lmd_1 (v_1^{bs})^{\otimes m} + \cdots +
\lmd_r  (v_r^{bs})^{\otimes m} - \mc{X}^{bs} \|_2^2,
\]
by the condition iii).
The tensors $(v_1^{bs})^{\otimes m}, \ldots,
(v_r^{bs})^{\otimes m}$ are linearly independent,
by the condition i). For $\eps>0$ small enough,
by \reff{vi:ls=bs+eps}, we can get
\[
 \|\lmd^{ls} - \lmd^{bs}\|_2 = O(\eps), \quad
\|  \mc{X}^{gp}   -  \mc{X}^{bs}   \| = O(\eps),
\]
\[
\|  \mc{F}   -  \mc{X}^{gp}   \|  \leq
\|  \mc{F}   -  \mc{X}^{bs} \| +
\|  \mc{X}^{bs}   -  \mc{X}^{gp}   \| = O(\eps).
\]
The constants in the above $O(\eps)$ eventually only depend on $\mF$ and $\xi$.
Because $\mc{X}^{opt}$ is improved from $\mc{X}^{gp}$
by solving the optimization problem \reff{apx:LS:u},
so $\|  \mc{F}   -  \mc{X}^{opt}   \| \leq
\|  \mc{F}   -  \mc{X}^{gp}   \| $.
Hence, \reff{err:F-Xrc:bs} is true.
\end{proof}

In Theorem~\ref{thm:lrkapx:err},
if $\eps=0$, then we can get $\mF = \mc{X}^{gp}$.

\begin{cor} \label{cor:apx=>TD}
Under the assumptions of Theorem~\ref{thm:lrkapx:err},
if $\rank_S(\mc{F})=r$, then $\mc{X}^{gp}$ produced by
Algorithm~\ref{alg:tsr:aprox} gives a rank decomposition for $\mF$.
\end{cor}

The assumptions in Theorem~\ref{thm:lrkapx:err} are often satisfied.
In contrast to the method in \cite{GPSTD},
Algorithm~\ref{alg:tsr:aprox} is more computationally tractable, because it
only requires to solve some linear least squares and Schur decompositions.
This is very efficient in applications, especially
for computing low rank decompositions for large tensors.
We refer to Example~\ref{sym:lwrk:STD}.

\section{Low rank approximations for nonsymmetric tensors}
\label{sc:lra:nonsym}
\setcounter{equation}{0}

Given a tensor $\mF \in \cpx^{n_1,\ldots,n_m}$ and an approximating rank $r$,
we want to find $\mc{X} \in \cpx^{n_1,\ldots,n_m}$
such that $\rank (\mc{X}) \leq r$ and $\mc{X}$
is close to $\mF$ as much as possible.
Note that $\rank (\mc{X}) \leq r$ if and only if
\[
\mc{X} = \sum_{s=1}^r  u^{s,1} \otimes \cdots \otimes u^{s,m},
\]
with each tuple $u^{(s)}:=(u^{s,1}, \ldots, u^{s,m}) \in
\cpx^{n_1} \times \cdots \times \cpx^{n_m}$.
This problem is equivalent to solving the nonlinear least squares
\be \label{opt:rank-r:F}
\min_{ u^{(1)}, \ldots,   u^{(m)} } \quad
\Big\| \sum_{s=1}^r  u^{s,1} \otimes \cdots \otimes u^{s,m}
- \mc{F} \Big\|^2.
\ee
It is typically hard to solve \reff{opt:rank-r:F} directly.
We propose to compute low rank approximations
by using generating polynomials, which are defined in \S\ref{sc:GR:nsym}.
It has three major stages:
\bit

\item Find an approximate generating matrix $G$ for $\mF$.

\item Construct the tuples $u^{(s)}$ from the computed $G$.

\item Improve the computed tuples $u^{(s)}$, if necessary.

\eit

\subsection{An algorithm for rank-$r$ approximation}

We consider nonsymmetric tensors in $\cpx^{n_1,\ldots,n_m}$.
Up to a permutation of indices, we can assume
\[
n_1 = \max\{n_1,\ldots, n_m\}.
\]
Recall the notation as in \reff{nota:r1:J}:
\[
r_1 \, := \,r - 1, \quad
J = \Big\{ (i,j,k):  \,\,
0\leq i \leq r_1, \, 2 \leq j \leq m, \, 1 \leq k \leq \on_j
\Big\}.
\]
We want to find $G \in \cpx^{[r_1] \times J}$ such that,
for each $\tau =(i,j,k) \in J$,
\[
\varphi[G, \tau] :=
\sum_{\ell=0}^{r_1}  G(\ell, \tau) x_{1,\ell}  -  x_{1,i} x_{j,k}
\]
is a generating polynomial for $\mF$,
or is close to such one as much as possible if it cannot.
Such a matrix $G$ is determined by the linear system
\be \label{F*w:aprx:Fxmu}
\sum_{\ell=0}^{r_1} \mc{F}_{x_{1,\ell} \cdot \mu }  G(\ell, \tau )
= \mc{F}_{x_{1,i} \cdot x_{j,k} \cdot \mu}
\quad \big(\,\mu \in \mathbb{M}_{ \{1,j\}^c } \big).
\ee
In the above, $\mF$ is indexed by multi-linear monomials
$x_{1,i_1} \cdots x_{m,i_m}$, as in \S\ref{sc:GR:nsym}.
If \reff{F*w:aprx:Fxmu} is satisfied, then
$G$ is a generating matrix for $\mF$ (cf.~Definition~\ref{df:GM:nsym}).
It is possible that \reff{F*w:aprx:Fxmu} is overdetermined and inconsistent.
For such case, we want $G$ to satisfy \reff{F*w:aprx:Fxmu}
as much as possible. This can be done by solving a linear least squares problem.
For each $\tau =(i,j,k) \in J$, denote
\be \label{df:AbF:ijk}
\left\{ \baray{rcl}
A[\mc{F},j] & := & \Big( \mc{F}_{x_{1,\ell} \cdot \mu}
\Big)_{\mu \in \mathbb{M}_{\{1,j\}^c}, 0\leq \ell \leq r_1}, \\
b[\mc{F},\tau] & := &  \Big( \mc{F}_{x_{1,i} \cdot x_{j,k} \cdot \mu }
\Big)_{\mu \in \mathbb{M}_{\{1,j\}^c} }.
\earay \right.
\ee
Then, \reff{F*w:aprx:Fxmu} can be equivalently expressed as
\[
A[\mc{F},j] \, G(:, \tau)  \,\, = \,\, b[\mc{F}, \tau],
\]
where $G(:, \tau)$ stands for the $\tau$-th column of $G$.
Let $G^{ls}$ be a minimizer of the linear least squares problem
\be \label{ls:AbF:ijk}
\min_{ G \in \cpx^{[r_1] \times J} } \quad
\sum_{ \tau =(i,j,k) \in J  }
\Big\| \,A[\mc{F}, j] \, G(:,\tau) - b[\mc{F},\tau] \,
\Big\|_2^2.
\ee
For each $(j,k)$, denote the matrix
{\smaller
\be \label{hatM:jk}
\hat{M}^{j,k} :=  \bbm
G^{ls}(0, (0,j,k)) & G^{ls}(1,(0,j,k)) & \cdots &  G^{ls}(r_1, (0,j,k)) \\
G^{ls}(0, (1,j,k)) & G^{ls}(1,(1,j,k)) & \cdots &  G^{ls}(r_1, (1,j,k)) \\
 \vdots & \vdots &  \ddots &  \vdots \\
G^{ls}(0, (r_1,j,k)) & G^{ls}(1,(r_1,j,k)) & \cdots &  G^{ls}(r_1, (r_1,j,k))
\ebm.
\ee \noindent}Choose
generic scalars $\xi_{j,k} >0$, and let
\be  \label{hatM(xi)}
\hat{M}(\xi) := \sum_{ (0,j,k) \in J } \,
\xi_{j,k} \hat{M}^{j,k} .
\ee
Then, compute the Schur Decomposition
\be  \label{sch:M(xi):nsy}
Q^* \hat{M}(\xi) Q \, = \, T,
\ee
where $Q=[q_1 \, \ldots \, q_r ]$ is unitary and
$T$ is upper triangular. Let
\be \label{hatv:s,j}
\hat{v}^{s,j} :=
\big(1, \,\, q_s^* \hat{M}^{j,1} q_s, \,\, \ldots, \, \,
q_s^* \hat{M}^{j,\on_j} q_s \big)
\ee
for $s =1,\ldots,r$ and $j=2,\ldots,m$.
Then, we solve the linear least squares
\be \label{scalLS:v1:vr}
\min_{ \substack{ z_1,\ldots, z_r \in \cpx^{n_1} } } \quad \left\|
\sum_{s=1}^r \, z_s \otimes \hat{v}^{s,2} \otimes
\cdots \otimes \hat{v}^{s,m} -\mc{F} \right \|^2.
\ee
Let $(\hat{v}^{1,1}, \ldots, \hat{v}^{\ell,1})$
be an optimal solution to \reff{scalLS:v1:vr}.

Then we construct the low rank tensor from the vectors $\hat{v}^{s,j}$ as
\be \label{Xgr:nsym}
\mc{X}^{gp} := \sum_{s=1}^r \hat{v}^{s,1} \otimes \hat{v}^{s,2} \otimes
\cdots \otimes \hat{v}^{s,m}.
\ee
If $\mF$ is close enough to $\sig_r(n_1,\ldots,n_m)$,
the set of $m$-tenors whose border ranks $\leq r$ (cf.~\S\ref{sbsc:symtsr}),
then $\mc{X}^{gp}$ is a good enough rank-$r$ approximation.
This is justified by Theorem \ref{thm:lraperr:nsym}.
The tensor $\mc{X}^{gp}$ is not mathematically guaranteed to
be a best rank-$r$ approximation. However, we can always improve it
by solving the optimization problem \reff{opt:rank-r:F},
with $(\hat{v}^{(1)}, \ldots, \hat{v}^{(r)})$ a starting point,
where each $\hat{v}^{(s)} := (\hat{v}^{s,1}, \ldots, \hat{v}^{s,m})$.

\medskip
Combining the above,
we get the following algorithm.

\begin{alg} \label{alg:lwrk:nsym}
Rank-$r$ approximations for nonsymmetric tensors. \\
\noindent
For given $\mc{F} \in \cpx^{n_1 \times \cdots n_m}$
and given $r \leq n_1 = \max\{n_1,\ldots,n_m\}$, do as follows:
\bit

\item [Step 1] Solve the linear least squares \reff{ls:AbF:ijk}
for an optimizer $G^{ls}$.

\item [Step 2] Choose generic $\xi_{j,k}>0$
(scale them as $\sum_{(0,j,k) \in J } \xi_{j,k} =1$) and
formulate $\hat{M}(\xi)$ as in \reff{hatM(xi)}.
Compute the Schur Decomposition \reff{sch:M(xi):nsy}.

\item [Step 3] Compute $\hat{v}^{s,j}$ as in \reff{hatv:s,j},
for $s=1,\ldots,r$ and $j=2,\ldots,m$.
Solve \reff{scalLS:v1:vr} for an optimizer
$(\hat{v}^{1,1}, \ldots, \hat{v}^{\ell,1})$.
Let $\hat{v}^{(s)} = (\hat{v}^{s,1}, \ldots, \hat{v}^{s,m})$ for each $s$.

\item [Step 4] Construct the tensor $\mc{X}^{gp}$ as in \reff{Xgr:nsym}.

\item [Step 5] Compute an improved solution $(\hat{u}^{(1)}, \ldots, \hat{u}^{(r)})$
of \reff{opt:rank-r:F} from the starting point $(\hat{v}^{(1)}, \ldots, \hat{v}^{(r)})$.
Let $\hat{u}^{(s)} = (\hat{u}^{s,1}, \ldots, \hat{u}^{s,m})$
for each $s$. Output
\be \label{Xopt:nsym}
\mc{X}^{opt} :=
\sum_{s=1}^r   \hat{u}^{s,1} \otimes \hat{u}^{s,2} \otimes
\cdots \otimes \hat{u}^{s,m}.
\ee

\eit

\end{alg}

In Algorithm~\ref{alg:lwrk:nsym}, the assumption $n_1 = \max\{n_1,\ldots,n_m\}$
can always be satisfied by permuting indices of tensors.

\begin{remark} \label{est:rk:nonsym}
The approximating rank $r$ in Algorithm~\ref{alg:lwrk:nsym} can be estimated
from the Catalecticant matrix $\mbox{Cat}(\mF)$.
If $\mF = \mc{X} + \mc{E}$, then
$\mbox{Cat}(\mF) = \mbox{Cat}(\mc{X}) + \mbox{Cat}(\mc{E})$.
If $\mc{E}$ is tiny, then $\rank \, \mbox{Cat}(\mc{X})$
can be estimated by $\rank\,\mbox{Cat}(\mF)$. By Lemma~\ref{rank=Cat:nsym},
$\rank\,(\mc{X}) = \rank\,\mbox{Cat}(\mc{X})$
if $\mc{X}$ is a generic point of $\sig_r(n_1,\ldots,n_m)$ and
$r$ is not bigger than the lower size of $\mbox{Cat}(\mc{X})$.
Hence, when $\mF$ is close to $\sig_r(n_1,\ldots,n_m)$, the value of
$r$ can be estimated by using the singular values of $\mbox{Cat}(\mF)$.
We refer to Remark~\ref{est:rk:cat}.
\end{remark}

When $r=1$, we can get a closed formula for $\mc{X}^{gp}$.
In such case, the column $\hat{G}(:, (0,j,k))$ is given as
\be \label{r=1:w0jk}
\hat{G}(:, (0,j,k)) =
\Big( \sum_{\mu \in \mathbb{M}_{\{1,j\}^c } } |\mc{F}_{\mu}|^2 \Big)^{-1}
\sum_{\mu \in \mathbb{M}_{\{1,j\}^c } }
\mc{F}_{\mu}^*  \mc{F}_{x_{j,k} \cdot \mu}  .
\ee
For $j=2,\ldots,m$, the vector $\hat{v}^{1,j}$ can be computed as
\be \label{hatv:1j}
\hat{v}^{1,j} = \Big(1, \,\, \hat{G}\big(0,(0,j,1)\big), \,\,
\ldots, \,\, \hat{G}\big(0,(0,j,\on_j)\big) \Big).
\ee
The vector $\hat{v}^{1,1}$ is then given by the formula
($k=1,\ldots,n_1$):
\be \label{r=1:hatv:11}
(\hat{v}^{1,1})_k =
\Big( \prod_{j=2}^m  \|\hat{v}^{1,j}\|_2^2 \Big)^{-1}
\sum_{i_2,\ldots,i_m} \mc{F}_{k,i_2,\ldots,i_m}
(\hat{v}^{1,2})^*_{i_2} \cdots (\hat{v}^{1,m})^*_{i_m}.
\ee

\subsection{Performance analysis}
\label{sbsc:err:nonsym}

We analyze the quality of low rank tensors produced by
Algorithm~\ref{alg:lwrk:nsym}. Suppose the tensor
\be \label{dc:Xbs:nsym}
\mc{X}^{bs} := \sum_{s=1}^r
u^{s,1} \otimes  \cdots  \otimes u^{s, m}
\ee
is a best, or an almost best, rank-$r$ approximation for $\mc{F}$. Let
\be \label{F=Xbs+E:nsym}
\mc{E} = \mF - \mc{X}^{bs}, \quad \| \mc{E} \| = \eps.
\ee
For convenience, we index $u^{s,j}$ by $k=0,1,\ldots,n_j-1$.
Assume $(u^{s,j})_0 \ne 0$ for $j =2,\ldots,m$.
This is satisfiable by applying a general linear transformation on tensors.
Up to a scaling, we can further assume that
\be \label{u:el:j2m:0=1}
(u^{s,2})_0 = \cdots = (u^{s,m})_0 = 1.
\ee
Let $A[\mF,j],\, b[\mc{F},\tau])$ be as in \reff{df:AbF:ijk}, then
\be \label{Ab:F=Xls+E:nsym}
\left\{ \baray{rcl}
A[\mc{F},j] &=& A[\mc{X}^{bs},j] + A[\mc{E},j], \\
b[\mc{F},\tau] &=& b[\mc{X}^{bs},\tau] + b[\mc{E},\tau].
\earay \right.
\ee
Denote $u^{(s)} := (u^{s,1},\ldots, u^{s,r})$ for $s=1,\ldots,r$.
The tuple $(u^{(1)},\ldots, u^{(r)})$ is called {\it scaling-optimal}
for $\mF$ if $(u^{1,1},\ldots, u^{r,1})$ is an optimizer of
\be \label{opt:scal:usj:nsym}
\min_{y_1, \ldots, y_r \in \cpx^{n_1}} \quad
\left \| \sum_{s=1}^r
y_s \otimes u^{s,2} \otimes \cdots \otimes u^{s,m}
  -\mc{F} \right \|_2^2.
\ee
It is reasonable to assume that $(u^{(1)},\ldots, u^{(r)})$ is scaling-optimal,
because otherwise we can replace it by a scaling-optimal one
in the choice of $\mc{X}^{bs}$.

The quality of the low rank tensors $\mc{X}^{gp}, \mc{X}^{opt}$,
produced by Algorithm~\ref{alg:lwrk:nsym}, is analyzed as follows. Denote
$\xi := (\xi_{j,k})_{(0,j,k) \in J}$.

\begin{theorem} \label{thm:lraperr:nsym}
Let $\mc{F}, \mc{X}^{bs}, \mc{E}, u^{s,j}$ be as in
\reff{dc:Xbs:nsym}-\reff{u:el:j2m:0=1}.
Assume that
\bit

\item [i)] $\big\{ (u^{s,1})_{0:r_1} \big\}_{s=1}^r$,
$\{ u^{s,2} \otimes \cdots \otimes u^{s,m} \}_{s=1}^r$
are linearly independent;

\item [ii)] each $A[\mc{X}^{bs},j] \,(2 \leq j \leq m)$ has full column rank;

\item [iii)] the tuple $(u^{(1)},\ldots, u^{(r)})$ is scaling-optimal for $\mF$;

\item [iv)] the following scalars are pairwisely distinct {\small
\be \label{dstnct:xijk:usjk}
\sum_{(0,j,k) \in J} \xi_{j,k} (u^{2,j})_k, \ldots,
\sum_{(0,j,k) \in J} \xi_{j,k} (u^{m,j})_k.
\ee}
\eit
If $\mc{X}^{gp}, \mc{X}^{opt}$ are from Algorithm~\ref{alg:lwrk:nsym}
and $\eps = \| \mc{E} \|$ is small enough, then
\be \label{err:F-Xgr:nsym}
\| \mc{X}^{bs} - \mc{X}^{gp} \| = O(\eps), \qquad
\| \mc{F} - \mc{X}^{opt} \| \leq \| \mc{F} - \mc{X}^{gp} \| = O(\eps),
\ee
where the constants in the above $O(\cdot)$ only depend on $\mF$ and $\xi$.
\end{theorem}
\begin{proof} \,
The main idea of the proof is similar to the one for Theorem~\ref{thm:lrkapx:err}.
By the condition i) and Theorem~\ref{thm:exist:wijk},
there exists a generating matrix $ G^{bs} \in \cpx^{[r_1] \times J}$ for $\mc{X}^{bs}$,
i.e., for each $\tau = (i,j,k) \in J$,
\[
A[\mc{X}^{bs},j] \, G^{bs}(:,\tau) = b[\mc{X}^{bs}, \tau].
\]
Since $G^{ls}, G^{bs}$ are the least squares solutions, we have
\[
G^{ls}(:,\tau) = A[\mc{F},j]^\dag \cdot b[\mc{F},\tau], \quad
G^{bs}(:,\tau) = A[\mc{X}^{bs},j]^\dag \cdot  b[\mc{X}^{bs},\tau].
\]
By \reff{F=Xbs+E:nsym}, we have
\[
\| \, A[\mc{F},j]- A[\mc{X}^{bs}, j] \,  \|_F \leq \eps, \quad
\| b[\mc{F},\tau] - b[\mc{X}^{bs}, \tau] \, \|_2 \leq \eps.
\]
Hence, by the condition ii), if $\eps>0$ is small enough, then,
\[
\| G^{ls}(:,\tau) - G^{bs}(:,\tau)  \|_2 = O(\epsilon).
\]
The constant in the above $O(\cdot)$ only depends on $\mF$.
Let $M_{bs}^{j,k}$ be the matrix obtained from $\hat{M}^{j,k}$ in \reff{hatM:jk}
by replacing $G^{ls}$ with $G^{bs}$, then
\[
\| \hat{M}^{j,k} - M_{bs}^{j,k} \|_F = O(\epsilon).
\]
Let
\[
M_{bs}(\xi) := \sum_{ (0,j,k)\in J }
\xi_{jk} M_{bs}^{j,k}.
\]
Since $\sum_{(0,j,k) \in J} \xi_{j,k} = 1$, we can further get
\[
 \| \hat{M}(\xi) - M_{bs}(\xi) \|_F = O(\epsilon).
\]
The constant in the above $O(\cdot)$ only depends on $\mF$ and $\xi$.

By \reff{dc:Xbs:nsym} and Proposition~\ref{prop:Mjk:eig},
for $s = 1, \ldots, r$ and $k \geq 1$, we have
\[
M_{bs}^{j,k}
\bbm  (u^{s,1})_0 \\ (u^{s,1})_1  \\ \vdots  \\(u^{s,1})_{r_1} \ebm
= (u^{s,j})_k
\bbm  (u^{s,1})_0 \\ (u^{s,1})_1  \\ \vdots  \\(u^{s,1})_{r_1} \ebm.
\]
So, each $(u^{s,j})_k$ is an eigenvalue
of $M_{bs}^{j,k} $, with the eigenvector $(u^{s,1})_{0:r_1}$,
for $s=1,\ldots,r$.
The matrices $M_{bs}^{j,k}$ are simultaneously diagonizable,
by the condition i),
and so is $M_{bs}(\xi)$. The eigenvalues of $M_{bs}(\xi)$ are
the numbers listed in \reff{dstnct:xijk:usjk}.
They are distinct from each other, by the condition iv).
For $\eps>0$ small enough, $\hat{M}(\xi)$ also
has distinct eigenvalues. For such $\eps>0$,
there exist unitary $Q_1$ and upper triangular $T_1$ such that
\be \label{SchErr:nsym}
Q_1^* M_{bs}(\xi) Q_1 = T_1, \quad
\| Q - Q_1\|_F = O(\eps), \quad \|T-T_1\|_F = O(\eps).
\ee
We refer to \cite{KPC94} or \cite[Theorem~4.1]{Sun95} for the above.
The constants in the above $O(\cdot)$ eventually only depend on $\mF$ and $\xi$.
Because $T_1$ has distinct diagonal entries,
there exists an upper triangular nonsingular matrix $R_1$ such that
$
\Lambda_1 \, := \, R_1^{-1} T_1 R_1
$
is diagonal. This results in the eigenvalue decomposition
\[
(Q_1R_1)^{-1} M_{bs}(\xi) (Q_1R_1) = \Lambda_1.
\]
Since all $M_{bs}^{j,k}$ are simultaneously diagonalizable, for each $(j,k)$,
there exist a nonsingular matrix $P$ and a diagonal matrix $D_{j,k}$ such that
\[
P^{-1} M_{bs}^{j,k} P = D_{j,k}, \quad
P^{-1} M_{bs}(\xi) P = \sum_{ (0,j,k) \in J } \xi_{j,k} D_{j,k}.
\]
Since $M_{bs}(\xi)$ does not have repeated eigenvalues,
each eigenvector is unique, up to a scaling.
So, there exists a diagonal matrix $D_0$ so that
$Q_1 R_1 = P D_0$.
For each $j,k$, the matrix
\[
Q_1^* M_{bs}^{j,k} Q_1 = R_1 D_0^{-1}  P^{-1} M_{bs}^{j,k}  P D_0 R_1^{-1}=
\]
\[
R_1 D_0^{-1} D_j D_0  R_1^{-1} \,
= \, R_1   D_j  R_1^{-1}  \overset{def}{=} \, T_j
\]
is upper triangular. The diagonals of $D_j$ and $T_j$ are same.
So, for all $s,j$,
\[
Q_1(:,s)^*M_{bs}^{j,k}  Q_1(:,s) =
P^{-1}(s,:)M_{bs}^{j,k}  P(:,s).
\]
($P^{-1}(s,:)$ denotes the $s$-th row of $P^{-1}$.)
For each $s$, the vector
\[
\bbm
P^{-1}(s,:)M_{bs}^{j,1} P(:,s) \\ \vdots \\ P^{-1}(s,:)M_{bs}^{j,\on_j}P(:,s))
\ebm
\]
is one of $(u^{1,j})_{1:\on_j}, \ldots, (u^{r, j})_{1:\on_j}$, for $j \geq 2$.
Up to a permutation of indices, we have the equation
\[
(u^{s, j})_{1:\on_j} =
\bbm
Q_1(:,s)^* M_{bs}^{j,1} Q_1(:,s) \\ \vdots \\
Q_1(:,s)^* M_{bs}^{j,\on_j} Q_1(:,s))
\ebm.
\]
By \reff{SchErr:nsym}, the above implies that
\be \label{usj=vsj+e}
u^{s, j} = \hat{v}^{s, j}  + O(\eps)
\ee
for all $s = 1, \ldots, r$, $j=2,\ldots,m$.
The constant in the above $O(\cdot)$ eventually only depends on $\mF$ and $\xi$.

For each $i=0,1,\ldots, n_1-1$,
the vector $\big( (u^{1,1})_i,\ldots, (u^{r,1})_i \big)$
is a minimizer of the linear least squares
\[
\min_{ (w_1, \ldots, w_r) \in \cpx^r  }
\Big \| \sum_{s=1}^r  \, w_s
\big(u^{s,2} \otimes \cdots \otimes u^{s,m}\big) - \mc{X}^{bs}(i,:) \Big\|_2^2,
\]
because $( u^{(1)}, \ldots, u^{(r)})$ is scaling optimal for $\mF$.
In Algorithm~\ref{alg:lwrk:nsym}, each vector
$\big((\hat{v}^{1,1})_i,\ldots, (\hat{v}^{r,1})_i\big)$
is a minimizer of the linear least squares
\[
\min_{ (w_1, \ldots, w_r) \in \cpx^r }
\Big\| \sum_{s=1}^r  \, w_s
\big(\hat{v}^{\ell,2} \otimes \cdots \otimes \hat{v}^{\ell,m} \big)
 - \mc{F}(i,:) \Big\|_2^2.
\]
In the above,  $\mc{X}^{bs}(i,:)$ (resp., $\mF(i,:)$)
denotes the sub-tensor of $\mc{X}^{bs}$ (resp., $\mF$)
whose first index is fixed to be $i$.
By the condition i), \reff{F=Xbs+E:nsym}, \reff{usj=vsj+e},
if $\eps>0$ is small enough, then for $s=1,\ldots,r$
\[
 \| u^{s, 1} - \hat{v}^{s,1}  \|_2 = O(\eps).
\]
Hence, it implies that
$
\| \mc{X}^{gp}   -  \mc{X}^{bs}   \| = O(\eps),
$
and
\[
\| \mc{F} - \mc{X}^{gp} \| \leq
\|  \mc{F}   -  \mc{X}^{bs} \| +
\|  \mc{X}^{bs}   -  \mc{X}^{gp}   \| = O(\eps).
\]
The constants in the above $O(\cdot)$ only depend on $\mF$ and $\xi$.
Note that $\mc{X}^{opt}$ is an improved approximation from the starting point $\mc{X}^{gp}$,
by solving the optimization problem \reff{opt:rank-r:F}. So,
$\| \mc{F} - \mc{X}^{opt} \| \leq \| \mc{F} - \mc{X}^{gp} \|$,
and \reff{err:F-Xgr:nsym} is true.
\end{proof}

In Theorem~\ref{thm:lraperr:nsym},
if $\eps=0$, then we can get $\mF = \mc{X}^{gp}$.

\begin{cor} \label{nsm:eps0=>TD}
Under the assumptions of Theorem~\ref{thm:lraperr:nsym},
if $\rank\,\mc{F}=r$, then the tensor $\mc{X}^{gp}$ produced by
Algorithm~\ref{alg:lwrk:nsym} gives a rank decomposition for $\mF$.
\end{cor}

The assumptions in Theorem~\ref{thm:lraperr:nsym}
are often satisfied. Because $\mc{X}^{gp}$ can be computed
by only solving linear least squares and Schur decompositions,
Algorithm~\ref{alg:lwrk:nsym} is very efficient
for computing low rank tensor decompositions,
especially for large scale tensors.
We refer to Example~\ref{emp:lowTD}.

\section{Numerical experiments}
\label{sc:comp}
\setcounter{equation}{0}

In this section, we present numerical experiments
for low rank tensor approximations.
The computation is implemented in MATLAB R2012a,
on a Lenovo Laptop with CPU@2.90GHz and RAM 16.0G.
In Step~5 of Algorithm~\ref{alg:tsr:aprox},
the MATLAB function {\tt lsqnonlin} is applied to solve
the nonlinear least squares problem \reff{apx:LS:u}.
In Step~5 of Algorithm~\ref{alg:lwrk:nsym},
{\tt lsqnonlin} is also applied to solve
the nonlinear least squares problem \reff{opt:rank-r:F}.

\subsection{Examples for symmetric tensors}
\label{sbsc:exmp:sym}

For the tensor $\mc{X}^{opt}= (u_1^{opt})^{\otimes m}
+ \cdots + (u_r^{opt})^{\otimes m}$ produced by Algorithm~\ref{alg:tsr:aprox},
we display it by showing the decomposing vectors
$u_1^{opt},\dots, u_r^{opt}$. Each $u_i^{opt}$ is displayed in a row,
separated by parenthesises. For neatness,
only four decimal digits are displayed.

\begin{exm}
Consider the cubic tensor $\mF \in \mt{S}^3( \cpx^n )$ such that
\[
\mF_{i_1 i_2 i_3} = \sin( i_1 + i_2 + i_3) \, \,
(1 \leq i_1, i_2, i_3 \leq n).
\]
For $n=6$, the first $3$ biggest singular values of
$\mbox{Cat}(\mF)$ are respectively
\[
   5.7857, \quad   5.4357, \quad  7 \times 10^{-16}.
\]
By Remark~\ref{est:rk:cat}, we consider the rank-$2$ approximation.
The tensor $\mc{X}^{opt}$
produced by Algorithm~\ref{alg:tsr:aprox} is given as:
{\tiny
\begin{verbatim}
(0.7053 - 0.3640i   0.0748 - 0.7902i  -0.6245 - 0.4899i  -0.7496 + 0.2608i  -0.1856 + 0.7717i   0.5491 + 0.5731i),
(0.7053 + 0.3640i   0.0748 + 0.7902i  -0.6245 + 0.4899i  -0.7496 - 0.2608i  -0.1856 - 0.7717i   0.5491 - 0.5731i).
\end{verbatim} \noindent}The approximation error
$\| \mF - \mc{X}^{opt} \| \approx  1 \times 10^{-15}$.
The approximating  tensor $\mc{X}^{opt}$
gives a rank decomposition, up to a tiny round-off error.
The computation is similar for other values of $n$.
\end{exm}

\begin{exm}
Consider the cubic tensor $\mF \in \mt{S}^3( \cpx^{n} )$ such that
\[
\mF_{i_1 i_2 i_3}  =  (i_1+i_2+i_3)^{-1} \quad (1 \leq i_1, i_2, i_3 \leq n).
\]
For $n=10$, the first $5$ biggest singular values of $\mbox{Cat}(\mF)$ are respectively
\[
    1.7660  ,\quad   0.1675   ,\quad   0.0135   ,\quad   0.0009   ,\quad   4 \times 10^{-5}.
\]
By Remark~\ref{est:rk:cat},
we consider rank-$r$ approximations with $r=1,2,3,4$.
The approximation errors produced by Algorithm~\ref{alg:tsr:aprox} are given as
\bcen
\begin{tabular}{|c|c|c|c|c|} \hline
$r$  &  1 & 2 & 3 & 4 \\  \hline
$\|\mF-\mc{X}^{opt}\|$    &    0.2632  &  0.0257  &  0.0021   &  0.0001   \\  \hline
\end{tabular}
\ecen
For convenience, we only list $\mc{X}^{opt}$ for $r=4$, which is given as:
{\scriptsize
\begin{verbatim}
(0.3946  0.3847  0.3749  0.3655  0.3564  0.3475  0.3388  0.3303  0.3220  0.3140),
(0.4885  0.4228  0.3664  0.3171  0.2743  0.2373  0.2054  0.1779  0.1539  0.1331),
(0.4851  0.3190  0.2091  0.1377  0.0910  0.0601  0.0396  0.0260  0.0170  0.0114),
(0.3453  0.1185  0.0417  0.0140  0.0041  0.0010  0.0002  0.0002  0.0002 -0.0000).
\end{verbatim}
}
\end{exm}

\begin{exm}
Consider the quartic tensor $\mF \in \mt{S}^4( \re^n )$ such that
\[
\mF_{i_1 i_2 i_3 i_4}  =   \exp(-i_1 i_2 i_3 i_4) \quad
(1 \leq i_1, i_2, i_3, i_4 \leq n).
\]
For $n=5$, the first $4$ biggest singular values of
$\mbox{Cat}(\mF)$ are respectively
\[
 0.4212  ,\quad  0.0365 ,\quad   0.0017  ,\quad  3 \times 10^{-5}.
\]
By Remark~\ref{est:rk:cat},
we consider rank-$r$ approximations with $r=1, 2,3$.
The approximation errors produced by Algorithm~\ref{alg:tsr:aprox} are given as:
\bcen
\begin{tabular}{|c|c|c|c|c|} \hline
$r$  &  1 &  2 & 3  \\  \hline
$\|\mF-\mc{X}^{opt}\|$    &   0.0809  & 0.0043  &  0.0001  \\  \hline
\end{tabular}
\ecen
For $r=3$, the approximating tensor $\mc{X}^{opt}$ is given as:
{\scriptsize
\begin{verbatim}
(1.4730-0.1984i   0.0646+0.0069i   0.0081+0.0034i   0.0026+0.0010i   0.0010+0.0003i),
(0.9753+0.9594i   0.0394-0.0347i  -0.0007-0.0106i  -0.0013-0.0016i  -0.0005-0.0002i),
(1.1326+0.5153i  -0.0454+0.0210i  -0.0086-0.0026i   0.0005-0.0019i   0.0008-0.0008i).
\end{verbatim}
}
\end{exm}

\begin{exm}
Consider the quartic tensor $\mF \in \mt{S}^4( \cpx^n )$ such that
\[
\mF_{i_1 i_2 i_3 i_4}  =    \log( i_1+i_2+i_3+i_4 ) \quad
(1 \leq i_1,i_2,i_3,i_4 \leq n).
\]
For $n=5$, the $5$ biggest singular values of $\mbox{Cat}(\mF)$ are respectively
\[
   36.9612  ,\quad  0.8344 ,\quad   0.0200 ,\quad   0.0005 ,\quad   1 \times 10^{-5}.
\]
By Remark~\ref{est:rk:cat},
we consider rank-$r$ approximations with $r=1,2,3,4$.
The approximation errors produced by Algorithm~\ref{alg:tsr:aprox} are given as:
\bcen
\begin{tabular}{|c|c|c|c|c|} \hline
$r$  &  1 & 2 & 3 & 4 \\  \hline
$\|\mF-\mc{X}^{opt}\|$    &     1.5806 &   0.0434   & 0.0013  &  $3 \times 10^{-5}$   \\  \hline
\end{tabular}
\ecen
For $r=4$, the approximating tensor $\mc{X}^{opt}$ is given as:
{\scriptsize
\begin{verbatim}
(0.2994-0.2994i  0.1297-0.1297i  0.0562-0.0562i  0.0242-0.0242i  0.0104-0.0104i),
(0.5197-0.5197i  0.3766-0.3766i  0.2729-0.2729i  0.1978-0.1978i  0.1434-0.1434i),
(0.7592-0.7592i  0.6986-0.6986i  0.6428-0.6428i  0.5914-0.5914i  0.5442-0.5442i),
(1.3204+0.0000i  1.3284+0.0000i  1.3364+0.0000i  1.3445+0.0000i  1.3527+0.0000i).
\end{verbatim}
}
\end{exm}

\begin{exm}
Consider the tensor $\mF \in \mt{S}^5( \re^n )$ such that
\[
\mF_{i_1 \ldots i_5} = \Big( i_1^2 + i_2^2 + i_3^2 + i_4^2 + i_5^2 \Big)^{1/2}
\]
for $1 \leq i_1,\ldots, i_5 \leq n$. For $n=4$,
the $5$ biggest singular values of $\mbox{Cat}(\mF)$ are respectively
\[
86.5310  ,\quad  3.5162   ,\quad   0.1215 ,\quad    0.0066   ,\quad   0.0003.
\]
By Remark~\ref{est:rk:cat},
we consider rank-$r$ approximations with $r=2,3,4$.
The approximation errors produced by Algorithm~\ref{alg:tsr:aprox} are given as:
\bcen
\begin{tabular}{|c|c|c|c|c|} \hline
$r$  &  2 & 3 & 4   \\  \hline
$\|\mF-\mc{X}^{opt}\|$    &   0.3760  &  0.0232  &  0.0014   \\  \hline
\end{tabular}
\ecen
For $r=4$, the approximating tensor $\mc{X}^{opt}$ is given as:
{\scriptsize
\begin{verbatim}
(0.5352 - 0.3888i   0.2098 - 0.1524i   0.0438 - 0.0318i   0.0036 - 0.0026i),
(0.7630 + 0.5543i   0.5756 + 0.4182i   0.3600 + 0.2616i   0.1870 + 0.1359i),
(1.5300 - 0.0000i   1.5435 - 0.0000i   1.5664 - 0.0000i   1.5990 - 0.0000i),
(1.1282 - 0.8197i   1.0731 - 0.7797i   0.9871 - 0.7172i   0.8781 - 0.6380i).
\end{verbatim}
}
\end{exm}

\begin{exm}
Consider the tensor $\mF \in \mt{S}^6( \re^n )$ such that
\[
\mF_{i_1 \ldots i_6} = \log \Big( \prod_{j=1}^6 i_j
+ \exp \big( \sum_{j=1}^6 i_j  \big )  \Big)
\]
for $1 \leq i_1,\ldots, i_6 \leq n$. For $n=4$,
the $4$ biggest singular values of $\mbox{Cat}(\mF)$ are respectively
\[
  306.8458  ,\quad  6.8405 ,\quad   0.0008  ,\quad    1 \times 10^{-7}.
\]
By Remark~\ref{est:rk:cat}, we consider rank-$r$ approximations with $r=2,3$.
The approximation errors produced by Algorithm~\ref{alg:tsr:aprox} are given as:
\bcen
\begin{tabular}{|c|c|c|} \hline
$r$  & 2 & 3   \\  \hline
$\|\mF-\mc{X}^{opt}\|$   &   0.0029   &   $3 \times 10^{-6} $   \\  \hline
\end{tabular}
\ecen
For $r=3$, the approximating tensor $\mc{X}^{opt}$ is given as:
{\scriptsize
\begin{verbatim}
(3.6696 - 0.0610i   3.6704 - 0.0609i   3.6711 - 0.0608i   3.6719 - 0.0608i),
(3.2072 + 1.7812i   3.2066 + 1.7807i   3.2060 + 1.7803i   3.2054 + 1.7799i),
(0.1842 - 0.3188i   0.1357 - 0.2348i   0.0752 - 0.1299i   0.0371 - 0.0640i).
\end{verbatim}
}
\end{exm}

\begin{exm}  \label{emp:rler:LRsTA}
We explore the practical performance of Algorithm~\ref{alg:tsr:aprox}.
By Theorem~\ref{thm:lrkapx:err}, if $\mF$ is sufficiently close to a rank-$r$
symmetric tensor, then $\mc{X}^{opt}$ is guaranteed to be a good enough approximation.
We verify this conclusion for random nearly low rank tensors.
First, generate a tensor
\[
\mc{R} =  (u_1)^{\otimes m} + \cdots + (u_r)^{\otimes m},
\]
where each $u_i \in \cpx^n$ has random real and imaginary parts,
obeying Gaussian distributions.
Then, perturb $\mc{R}$ by a small tensor $\mc{E} \in \mt{S}^m(\cpx^n)$,
whose entries are also randomly generated.
We scale $\mc{E}$ to have a desired norm $\eps >0$. Finally, set
\[
\mF = \mc{R} + \mc{E}.
\]
The approximation quality of $\mc{X}^{opt}$ can be measured by the relative error
\[
{\tt relerr} :=  \| \mF - \mc{X}^{opt} \| \,\, / \,\, \| \mc{E} \|.
\]
For each $(n,m,r)$ from Table~\ref{tab:rlerr:sym} and
each $\eps$ among $10^{-1}, 10^{-2}, 10^{-3}$,
we generate $20$ random instances. For each instance,
apply Algorithm~\ref{alg:tsr:aprox} to get a rank-$r$ approximation.
For each $(n,m,r)$, we list the maximum value of the occurring {\tt relerr}
(denoted as {\tt mrlerr}),
and the average of consumed time (in seconds), in Table~\ref{tab:rlerr:sym}.
%
%
\bcen
\begin{table}[htb]
\caption{Performance of Algorithm~\ref{alg:tsr:aprox}
for computing low rank symmetric tensor approximations.}
\btab{||r|r|c|r||r|r|c|r||} \hline
$(n,m)$   & $r$   & {\tt mrlerr}  & {\tt time}  & $(n,m)$ & $r$ & {\tt mrlerr} &  {\tt time}  \\ \hline
(50, 3) & 1   &   0.9991   & 27.54 &  (30, 4) & 1 &  0.9998  & 65.99   \\  \hline
(40, 3) & 2   &   0.9975   & 18.74 &  (25, 4) & 2 &  0.9992  & 46.89   \\  \hline
(30, 3) & 3   &   0.9934   & 10.13 &  (20, 4) & 3 &  0.9981  & 27.92   \\  \hline
(20, 3) & 4   &   0.9817   &  3.71 &  (15, 4) & 4 &  0.9936  & 11.39   \\  \hline
(10, 3) & 5   &   0.9094   &  0.53 &  (10, 4) & 5 &  0.9772  &  2.44   \\  \hline
\etab
\label{tab:rlerr:sym}
\end{table}
\ecen
\end{exm}

\begin{exm} \label{sym:lwrk:STD}
(Low rank symmetric tensor decompositions.)
By Corollary~\ref{cor:apx=>TD}, if $\rank_S(\mF) =r$, then $\mc{X}^{gp}$
produced by Algorithm~\ref{alg:tsr:aprox} gives a rank decomposition for $\mF$,
under the conditions in Theorems~\ref{thm:lrkapx:err}.
We verify this fact for random low rank tensors.
We generate $\mF$ in the same way as in Example~\ref{emp:rler:LRsTA}, except letting $\mc{E}=0$.
For each generated $\mF$, we apply Algorithm~\ref{alg:tsr:aprox}
with $r = \rank \, \mbox{Cat}(\mF)$. This is justified by Remark~\ref{est:rk:cat}.
For each $(n,m,r)$ in Table~\ref{tab:time:lwrkSTD},
we generate $20$ random instances of $\mF$.
For all generated instances, the approximating tensors $\mc{X}^{gp}$ gave
correct rank decompositions (up to tiny round-off errors in the order around $10^{-10}$).
Since $\mc{X}^{gp} = \mF$, Step~5 of Algorithm~\ref{alg:tsr:aprox}
does not need to be performed.
For each $(n,m,r)$, the average time (in seconds) consumed by the computation is listed in
Table~\ref{tab:time:lwrkSTD}. As we can see, low rank decompositions
can be computed very efficiently.
Algorithm~\ref{alg:tsr:aprox} only needs to solve some linear least squares
and Schur's decompositions. So, it is suitable
for computing low rank symmetric tensor decompositions,
especially for large tensors.
\bcen
\begin{table}[htb]
\caption{Performance of Algorithm~\ref{alg:tsr:aprox}
for computing low rank decompositions of symmetric tensors.}
\btab{||c|r|r||c|r|r||c|r|r||} \hline
$(n,m)$ & $r$  & {\tt time} & $(n,m)$ & $r$ & {\tt time} &  $(n,m)$ & $r$ & {\tt time} \\ \hline
(10, 3) &  5  &  0.43     &  (10, 4) & 5  &  0.43    &  (5, 5)  &  10  & 0.57   \\  \hline
(20, 3) &  10 &  7.39     &  (15, 4) & 10  & 3.50    &  (10, 5)  & 15  & 29.12  \\  \hline
(30, 3) &  15  & 38.49   &  (20, 4) & 15  & 13.27   &  (15, 5)  & 20  & 311.07   \\  \hline
(40, 3) &  20  & 124.04   & (25, 4) & 20  & 28.34   &  (5, 6)  &  10  &  2.85  \\  \hline
(50, 3) & 25 &  307.83  &  (30, 4) & 25  & 53.99   &  (10, 6)  & 20  &  182.89   \\  \hline
\etab
\label{tab:time:lwrkSTD}
\end{table}
\ecen
\end{exm}

\subsection{Examples for nonsymmetric tensors}
\label{ssc:exmp:nonsym}
\setcounter{equation}{0}

We display the tensor $\mc{X}^{opt}$,
given as in \reff{Xopt:nsym} by Algorithm~\ref{alg:lwrk:nsym},
by displaying the tuples $\hat{u}^{(s)}$ one by one.
Different tuples are separated by a blank row.
For each tuple $\hat{u}^{(s)} = (\hat{u}^{s,1}, \ldots, \hat{u}^{s,m})$,
we display the vectors $\hat{u}^{s,1}, \ldots, \hat{u}^{s,m}$
row by row in the order.

\begin{exm}
Consider the tensor $\mF \in \cpx^{7 \times 6 \times 5}$ such that
\[
\mF_{i_1,i_2,i_3}  = \Big( \exp(i_1) + \exp(i_2^2) + \exp(i_3^3) \Big)^{-1}
\]
for all $i_1,i_2,i_3$ in the range.
The $3$ biggest singular values of $\mbox{Cat}(\mF)$ are:
\[
 0.1542  ,\quad  0.0010 ,\quad 7  \times 10^{-12}.
\]
By Remark~\ref{est:rk:nonsym}, we consider rank-$r$ approximations with $r=1,2$.
Applying Algorithm~\ref{alg:lwrk:nsym}, we get the approximation errors:
\bcen
\begin{tabular}{|c|c|c|} \hline
$r$ &  1 & 2 \\  \hline
 $\|\mF-\mc{X}^{opt}\|$ &    $0.0107$  &  $5 \times 10^{-4}$  \\ \hline
\end{tabular}
\ecen
For $r=2$, the rank-$2$ approximating tensor $\mc{X}^{opt}$ is given as:
{\scriptsize
\begin{verbatim}
(0.1157    0.0682    0.0286    0.0083    0.0017    0.0003    0.0000),
(1.0000    0.0851   -0.0001   -0.0000   -0.0000   -0.0000),
(1.0000    0.0000   -0.0000    0.0000    0.0000);

(0.0069    0.0098    0.0106    0.0083    0.0048    0.0022    0.0009),
(1.0000    0.9887    0.0161    0.0000    0.0000    0.0000),
(1.0000    0.0403    0.0000   -0.0000   -0.0000).
\end{verbatim}
}
\end{exm}

\begin{exm}
Consider the tensor $\mF \in \cpx^{5 \times 4\times 4}$ such that
\[
\mF_{i_1,i_2,i_3}  =   \cos(i_1-i_2-i_3)
\]
for all $i_1,i_2,i_3$ in the range.
The $3$ biggest singular values of $\mbox{Cat}(\mF)$ are:
\[
5.0371  ,\quad   3.8638    ,\quad   2 \times 10^{-16}.
\]
By Remark~\ref{est:rk:nonsym}, we consider the rank-$2$ approximation.
The computed rank-$2$ approximating tensor $\mc{X}^{opt}$ is given as:
{\scriptsize
\begin{verbatim}
(0.2702 + 0.4207i   0.5000 - 0.0000i   0.2702 - 0.4207i  -0.2081 - 0.4546i  -0.4950 - 0.0706i),
(1.0000             0.5403 + 0.8415i  -0.4161 + 0.9093i  -0.9900 + 0.1411i),
(1.0000             0.5403 + 0.8415i  -0.4161 + 0.9093i  -0.9900 + 0.1411i);

(0.2702 - 0.4207i   0.5000 - 0.0000i   0.2702 + 0.4207i  -0.2081 + 0.4546i  -0.4950 + 0.0706i),
(1.0000             0.5403 - 0.8415i  -0.4161 - 0.9093i  -0.9900 - 0.1411i),
(1.0000             0.5403 - 0.8415i  -0.4161 - 0.9093i  -0.9900 - 0.1411i).
\end{verbatim} \noindent}The approximation error
$\| \mF - \mc{X}^{opt} \|  \approx   6  \times 10^{-16}$.
It gives a rank decomposition for the tensor, up to a tiny round-off error.
\end{exm}

\begin{exm}
Consider the tensor $\mF \in \cpx^{8 \times 7 \times 6 \times 5}$
such that
\[
\mF_{i_1,i_2,i_3,i4}  = (1+i_1+2 i_2+3 i_3+4 i_4)^{-1}
\]
for all $i_1,i_2,i_3,i_4$ in the range.
The $5$ biggest singular values of $\mbox{Cat}(\mF)$ are:
\[
    1.2758  ,\quad  0.0585  ,\quad  0.0030  ,\quad  0.0001  ,\quad   5 \times 10^{-6}.
\]
By Remark~\ref{est:rk:nonsym}, we consider rank-$r$ approximations with $r=1,2,3,4$.
Applying Algorithm~\ref{alg:lwrk:nsym}, we get the approximation errors:
\bcen
\begin{tabular}{|c|c|c|c|c|} \hline
$r$ &  1 & 2 & 3  &  4\\  \hline
 $\|\mF-\mc{X}^{opt}\|$ &    0.0690  &  0.0041  &  0.0002   &   $9 \times 10^{-6}$   \\ \hline
\end{tabular}
\ecen
For $r=2$, the approximating tensor $\mc{X}^{opt}$ is given as:
{\scriptsize
\begin{verbatim}
(0.0471    0.0411    0.0361    0.0317    0.0279    0.0246    0.0217    0.0192),
(1.0000    0.7662    0.5938    0.4632    0.3620    0.2821    0.2181),
(1.0000    0.6750    0.4654    0.3227    0.2219    0.1485),
(1.0000    0.5971    0.3669    0.2242    0.1309);

(0.0421    0.0413    0.0406    0.0398    0.0390    0.0382    0.0375    0.0368),
(1.0000    0.9639    0.9268    0.8904    0.8555    0.8224    0.7912),
(1.0000    0.9452    0.8897    0.8374    0.7892    0.7455),
(1.0000    0.9259    0.8533    0.7879    0.7302).
\end{verbatim}
}
\end{exm}

\begin{exm}
Consider the tensor $\mF \in \cpx^{5 \times 5 \times 4 \times 4}$ such that
\[
\mF_{i_1,i_2,i_3,i4}  = \cos(i_1+i_2-i_3-i_4)
- 10^{-3} \cdot \sin(i_1 i_2 i_3 i_4)
\]
for all $i_1,i_2,i_3,i_4$ in the range.
The $3$ biggest singular values of $\mbox{Cat}(\mF)$ are:
\[
 10.2674  ,\quad  9.7136  ,\quad   0.0059.
\]
By Remark~\ref{est:rk:nonsym}, we consider the rank-$2$ approximation.
Applying Algorithm~\ref{alg:lwrk:nsym},
we get the rank-$2$ approximating tensor $\mc{X}^{opt}$ given as:
{\scriptsize
\begin{verbatim}
(0.4998 - 0.0001i   0.2701 - 0.4206i  -0.2081 - 0.4545i  -0.4948 - 0.0704i  -0.3265 + 0.3783i),
(1.0000             0.5406 - 0.8415i  -0.4161 - 0.9095i  -0.9901 - 0.1412i  -0.6536 + 0.7569i),
(1.0000             0.5405 + 0.8417i  -0.4163 + 0.9094i  -0.9901 + 0.1411i),
(1.0000             0.5405 + 0.8417i  -0.4163 + 0.9094i  -0.9901 + 0.1411i);

(0.4998 + 0.0001i   0.2701 + 0.4206i  -0.2081 + 0.4545i  -0.4948 + 0.0704i  -0.3265 - 0.3783i),
(1.0000             0.5406 + 0.8415i  -0.4161 + 0.9095i  -0.9901 + 0.1412i  -0.6536 - 0.7569i),
(1.0000             0.5405 - 0.8417i  -0.4163 - 0.9094i  -0.9901 - 0.1411i),
(1.0000             0.5405 - 0.8417i  -0.4163 - 0.9094i  -0.9901 - 0.1411i).
\end{verbatim} \noindent}The approximation error
$ \| \mF - \mc{X}^{opt} \| \approx  0.0141$.
\end{exm}

\begin{exm}
Consider the tensor $\mF \in \cpx^{9 \times  8  \times  7 \times  6  \times  5  }$
such that
\[
\mF_{i_1,i_2,i_3,i_4, i_5}  =  \arctan
\Big(i_1 \cdot (i_2)^2 \cdot (i_3)^3 \cdot (i_4)^4 \cdot (i_5)^5 \Big)
\]
for all $i_1,i_2,i_3,i_4, i_5$ in the range.
The first $4$ biggest singular values of $\mbox{Cat}(\mF)$ are
\[
  193.1060  ,\quad  1.0818  ,\quad  0.0089  ,\quad   9 \times 10^{-7}.
\]
By Remark~\ref{est:rk:nonsym}, we consider rank-$r$ approximations with $r=1,2,3$.
Applying Algorithm~\ref{alg:lwrk:nsym}, we get the approximation errors:
\bcen
\begin{tabular}{|c|c|c|c|} \hline
$r$ &  1 & 2 & 3   \\  \hline
 $\|\mF-\mc{X}^{opt}\|$ &   1.1127  &  0.0349  &  0.0001 \\ \hline
\end{tabular}
\ecen
For $r=3$, the approximating tensor $\mc{X}^{opt}$ is given as:
{\scriptsize
\begin{verbatim}
( 1.5708    1.5708    1.5708    1.5708    1.5708    1.5708    1.5708    1.5708    1.5708),
( 1.0000    1.0000    1.0000    1.0000    1.0000    1.0000    1.0000    1.0000),
( 1.0000    1.0000    1.0000    1.0000    1.0000    1.0000    1.0000),
( 1.0000    1.0000    1.0000    1.0000    1.0000    1.0000),
( 1.0000    1.0000    1.0000    1.0000    1.0000);

( 0.2161    0.0373    0.0120    0.0053    0.0028    0.0016    0.0011    0.0007    0.0005),
( 1.0000    0.0249    0.0024    0.0005    0.0002    0.0001    0.0000    0.0000),
( 1.0000    0.0033    0.0001    0.0000    0.0000   -0.0000   -0.0000),
( 1.0000    0.0005    0.0000   -0.0000   -0.0000   -0.0000),
( 1.0000    0.0001   -0.0000   -0.0000   -0.0000);

(-1.0015   -0.5010   -0.3338   -0.2502   -0.2002   -0.1668   -0.1430   -0.1251   -0.1112),
( 1.0000    0.2500    0.1110    0.0624    0.0400    0.0277    0.0204    0.0156),
( 1.0000    0.1249    0.0370    0.0156    0.0080    0.0046    0.0029),
( 1.0000    0.0624    0.0123    0.0039    0.0016    0.0008),
( 1.0000    0.0312    0.0041    0.0010    0.0003).
\end{verbatim}
}
\end{exm}

\begin{exm}
Consider the  tensor $\mF \in \cpx^{  5  \times  5  \times
 5   \times    4   \times    4   \times    4 }$
such that
\[
\mF_{i_1,i_2,i_3,i_4, i_5, i_6}  =  \log \big(1+\exp(i_1 i_2i_3 + i_4i_5 i_6) \big)
\]
for all $i_1, \ldots, i_6$ in the range.
The first $4$ biggest singular values of $\mbox{Cat}(\mF)$ are
\[
10^3 \times (   4.5208  ,\quad  0.5794  ,\quad  0.0002  ,\quad  2 \times 10^{-5}  ).
\]
By Remark~\ref{est:rk:nonsym}, we consider rank-$r$ approximations with $r=2,3$.
Applying Algorithm~\ref{alg:lwrk:nsym}, we get the approximation errors:
\bcen
\begin{tabular}{|c|c|c|} \hline
$r$  & 2 & 3\\  \hline
 $\|\mF-\mc{X}^{opt}\|$ &    0.1919   &  0.0365  \\ \hline
\end{tabular}
\ecen
For $r=3$, the approximating tensor $\mc{X}^{opt}$ is given as
{\scriptsize
\begin{verbatim}
(1.0000    1.0000    1.0000    1.0000    1.0000),
(1.0000    1.0000    1.0000    1.0000    1.0000),
(1.0000    1.0000    1.0000    1.0000    1.0000),
(1.0000    2.0000    3.0000    4.0000),
(1.0000    2.0000    3.0000    4.0000),
(1.0000    2.0000    3.0000    4.0000);

(1.0000    2.0000    3.0000    4.0000    5.0000),
(1.0000    2.0000    3.0000    4.0000    5.0000),
(1.0000    2.0000    3.0000    4.0000    5.0000),
(1.0000    1.0000    1.0000    1.0000),
(1.0000    1.0000    1.0000    1.0000),
(1.0000    1.0000    1.0000    1.0000);

(0.1346    0.0435    0.0154    0.0056    0.0021),
(1.0000    0.3229    0.1141    0.0418    0.0158),
(1.0000    0.3229    0.1141    0.0418    0.0158),
(1.0000    0.3233    0.1148    0.0430),
(1.0000    0.3233    0.1148    0.0430),
(1.0000    0.3233    0.1148    0.0430).
\end{verbatim}}
\end{exm}

\begin{exm} \label{exmp:rlerr:LRTA}
We explore the practical performance of Algorithm~\ref{alg:lwrk:nsym}.
Theorem~\ref{thm:lraperr:nsym} shows that $\mc{X}^{opt}$
is a good enough rank-$r$ approximation if $\mF$ is sufficiently close to a rank-$r$ tensor.
We verify this conclusion for random nearly low rank tensors.
For given $(n_1,\ldots,n_m)$ and $r$, we generate a tensor
\[
\mc{R} = \sum_{s=1}^r u^{s,1} \otimes u^{s,2} \otimes \cdots \otimes u^{s,m},
\]
where each $u^{s,j} \in \cpx^{n_j}$ is a complex vector
whose real and imaginary parts are generated randomly, obeying Gaussian distributions.
Then, we perturb $\mc{R}$ by a small tensor $\mc{E}$,
whose are also randomly generated in the same way.
Scale $\mc{E}$ to have a desired norm $\eps$, and then let
\[
\mF = \mc{R} + \mc{E}.
\]
We choose $\eps$ among $10^{-1}, 10^{-2}, 10^{-3}$,
and use the relative error
\[
{\tt relerr}  =  \| \mF - \mc{X}^{opt} \| \,\, / \,\,  \| \mc{E} \|
\]
to measure the approximation quality of $\mc{X}^{opt}$.
For each $(n_1,\ldots,n_m)$, $r$ and such $\eps$,
we generate $20$ instances of $\mc{R},\mF,\mc{E}$ as above,
and then apply Algorithm~\ref{alg:lwrk:nsym} to compute $\mc{X}^{opt}$.
The computational results are reported in Table~\ref{tb:rlerr:nsymLRTA}.
For each case of $(n_1,\ldots,n_m)$ and $r$,
we list the maximum of the occuring relative errors
(denoted as {\tt mrlerr}),
and the average time (in seconds) consumed by the computation.
%
%
\bcen
\begin{table}
\caption{Performance of Algorithm~\ref{alg:lwrk:nsym}
for computing low rank tensor approximations. }
\btab{||l|l|r|r||l|l|r|r||} \hline
$(n_1, n_2, n_3)$ & $r$  &  {\tt mrlerr} &  {\tt time}  &
           $(n_1, n_2, n_3, n_4)$    & $r$  &  {\tt mrlerr} & {\tt time}  \\ \hline
(50,50,50) & 1  &  0.9995  & 10.67   & (30,30,30,30) & 1 &  0.9999  &  63.71  \\  \hline
(40,40,40) & 2  &  0.9984  & 16.95   & (25,25,25,25) & 2 &  0.9998  &  69.88  \\  \hline
(30,30,30) & 3  &  0.9952  &  9.20   & (20,20,20,20) & 3 &  0.9993  &  41.39  \\  \hline
(20,20,20) & 4  &  0.9865  &  2.64   & (20,20,15,15) & 4 &  0.9984  &  28.47  \\  \hline
(10,10,10) & 5  &  0.9356  &  0.47   & (15,15,10,10) & 5 &  0.9961  &   4.46 \\  \hline
\etab
\label{tb:rlerr:nsymLRTA}
\end{table}
\ecen
\end{exm}

\begin{exm} \label{emp:lowTD}
(Low rank decompositions of nonsymmetric tensors.)
By Corollary~\ref{nsm:eps0=>TD}, if $\rank \mF = r$,
then the tensor $\mc{X}^{gp}$ produced by Algorithm~\ref{alg:lwrk:nsym}
gives a rank decompositions for $\mF$, under the conditions of
Theorem~\ref{thm:lraperr:nsym}. We verify this conclusion
for random low rank tensors. We generate $\mF$
in the same way as in Example~\ref{exmp:rlerr:LRTA}, except $\mc{E}=0$.
In the computation, we choose $r = \rank\,\mbox{Cat}(\mF)$,
as pointed out in Remark~\ref{est:rk:nonsym}.
For each case of $(n_1,\ldots,n_m)$ and $r$,
we generate $20$ instances of $\mF$.
(The symbol $\star$ means that only $10$ instances
are generated, because of the longer computational time).
For all the instances, the approximating tensors $\mc{X}^{gp}$ gave
correct rank decompositions (up to round-off errors in the order around $10^{-10}$).
The Step~5 in Algorithm~\ref{alg:lwrk:nsym} does not need to be performed,
because $\mF = \mc{X}^{gp}$. For each case of $(n_1,\ldots,n_m)$ and $r$,
we list the average of the computational time (in seconds),
in Table~\ref{timetab:lowTD}. It shows that Algorithm~\ref{alg:lwrk:nsym} is efficient
for computing low rank tensor decompositions, especially for large tensors.
This is because it only needs to solve some linear least squares
and Schur's decompositions.
\bcen
\begin{table}
\caption{Performance of Algorithm~\ref{alg:lwrk:nsym} for computing
low rank decompositions of nonsymmetric tensors.}
\btab{||l|l|r||l|l|r||} \hline
$(n_1, n_2, n_3)$ & $r$  & {\rm time}  &
          $(n_1, n_2, n_3, n_4)$  & $r$  & {\rm time}      \\ \hline
(60,60,60)    &  10  &  0.47  & (20,20,20,20) &  10 &    2.55  \\  \hline
(70,70,70)    &  20  &  2.28  & (25,25,25,25) &  20 &   18.00  \\  \hline
(80,80,80)    &  30  &  5.51  & (40,30,25,20) &  30 &   36.63  \\  \hline
(90,90,90)    &  40  & 13.90  & (50,40,30,25) &  40 &  208.00    \\  \hline
(100,100,100) &  50  & 25.96  & (60,50,40,30)$\star$ & 50 &  1010.93 \\  \hline
\etab
\label{timetab:lowTD}
\end{table}
\ecen
\end{exm}

\noindent
{\bf Acknowledgement}
The research was partially supported by the NSF grants
DMS-0844775 and DMS-1417985.

\end{document}